\documentclass[11pt]{amsart}
\usepackage{amsopn, amssymb, amscd, amsmath, amsthm}
\usepackage{graphicx, graphics, epsfig}
\usepackage{enumerate}

\usepackage{todonotes} 

\usepackage{lscape, pdflscape}  
\usepackage{array, multirow} 
\usepackage{afterpage, capt-of} 

\allowdisplaybreaks

\newcommand{\nc}{\newcommand}

\nc{\Gl}{\mathsf{GL}} \nc{\Or}{\mathsf{O}}  \nc{\SO}{\mathsf{SO}}   \nc{\Sl}{\mathsf{SL}}
\nc{\G}{\mathsf{G}} \nc{\K}{\mathsf{K}}  \nc{\T}{\mathsf{T}} \nc{\Lsf}{\mathsf{L}}
\nc{\Qb}{\mathsf{Q}_\Beta} \nc{\Hb}{\mathsf{H}_\Beta} \nc{\Ub}{\mathsf{U}_\Beta}
\nc{\Gb}{\mathsf{G}_\Beta} \nc{\Kb}{\mathsf{K}_\Beta}
\nc{\PPP}{\mathsf{P}}
\nc{\U}{\mathsf{U}} \nc{\N}{\mathsf{N}}
\nc{\Ss}{\mathsf{S}}

\nc{\laH}{\la\!\la} \nc{\raH}{\ra\!\ra}

\nc{\ipH}{{\laH \cdot, \cdot \raH}}

\nc{\Vg}{{V(\ggo)}}

\nc{\alert}{\color{blue}}

\nc{\fg}{\mathfrak{f}}  \nc{\vg}{\mathfrak{v}} \nc{\wg}{\mathfrak{w}} \nc{\zg}{\mathfrak{z}} \nc{\ngo}{\mathfrak{n}} \nc{\kg}{\mathfrak{k}} \nc{\mg}{\mathfrak{m}} \nc{\bg}{\mathfrak{b}} \nc{\ggo}{\mathfrak{g}} \nc{\ggob}{\overline{\mathfrak{g}}} \nc{\sog}{\mathfrak{so}} \nc{\sug}{\mathfrak{su}} \nc{\spg}{\mathfrak{sp}} \nc{\slg}{\mathfrak{sl}} \nc{\glg}{\mathfrak{gl}} \nc{\cg}{\mathfrak{c}} \nc{\rg}{\mathfrak{r}}  \nc{\hg}{\mathfrak{h}} \nc{\tgo}{\mathfrak{t}} \nc{\ug}{\mathfrak{u}} \nc{\dg}{\mathfrak{d}} \nc{\ag}{\mathfrak{a}} \nc{\pg}{\mathfrak{p}} \nc{\sg}{\mathfrak{s}} \nc{\affg}{\mathfrak{aff}} \nc{\qg}{\mathfrak{q}}
\nc{\Xg}{\mathfrak{X}} \nc{\lgo}{\mathfrak{l}}

\nc{\pca}{\mathcal{P}} \nc{\nca}{\mathcal{N}} \nc{\lca}{\mathcal{L}} \nc{\oca}{\mathcal{O}} \nc{\mca}{\mathcal{M}} \nc{\tca}{\mathcal{T}} \nc{\aca}{\mathcal{A}} \nc{\cca}{\mathcal{C}} \nc{\gca}{\mathcal{G}} \nc{\sca}{\mathcal{S}} \nc{\hca}{\mathcal{H}} \nc{\bca}{\mathcal{B}} \nc{\dca}{\mathcal{D}}

\nc{\vp}{\varphi} \nc{\ddt}{\tfrac{{\rm d}}{{\rm d}t}} \nc{\dds}{\tfrac{{\rm d}}{{\rm d}s}} \nc{\ddtbig}{\frac{{\rm d}}{{\rm d}t}} \nc{\dd}{{\rm d}}
\nc{\dpar}{\tfrac{\partial}{\partial t}} \nc{\im}{\mathtt{i}}
\renewcommand{\Re}{{\rm Re}}\renewcommand{\Im}{{\rm Im}}

\nc{\RR}{{\mathbb R}} \nc{\HH}{{\mathbb H}} \nc{\CC}{{\mathbb C}} \nc{\ZZ}{{\mathbb Z}}
\nc{\FF}{{\mathbb F}} \nc{\NN}{{\mathbb N}} \nc{\QQ}{{\mathbb Q}} \nc{\PP}{{\mathbb P}}

\nc{\vs}{\vspace{.2cm}} \nc{\vsp}{\vspace{1cm}} \nc{\ip}{{\langle\cdot,\cdot\rangle}}
\nc{\ipp}{(\cdot,\cdot)} \nc{\la}{\langle} \nc{\ra}{\rangle} \nc{\unm}{\tfrac{1}{2}}
\nc{\unc}{\tfrac{1}{4}} \nc{\und}{\tfrac{1}{16}} \nc{\no}{\vs\noindent}
\nc{\lam}{\Lambda^2(\RR^n)^*\otimes\RR^n} \nc{\tangz}{{\rm T}^{\rm Zar}}
\nc{\lamg}{\Lambda^2\ggo^*\otimes\ggo}
\nc{\nor}{{\sf n}}  \nc{\mum}{/\!\!/} \nc{\kir}{/\!\!/\!\!/}
\nc{\Ri}{\tfrac{4\Ric_{\mu}}{||\mu||^2}} \nc{\ds}{\displaystyle}
\nc{\lb}{[\cdot,\cdot]} \nc{\isn}{\tfrac{1}{||v||^2}}
\nc{\gkp}{(\ggo=\kg\oplus\pg,\ip)} \nc{\ukh}{(\ug=\kg\oplus\hg,\ip)}
\nc{\tgkp}{(\tilde{\ggo}=\kg\oplus\pg,\ip)}
\nc{\wt}{\widetilde}
\nc{\raw}{\rightarrow} \nc{\lraw}{\longrightarrow} \nc{\hqn}{\mathcal{H}_{q,n}}
\newcommand{\minimatrix}[4]{\left(\begin{smallmatrix} {#1} & {#2} \\ {#3} & {#4} \end{smallmatrix}\right)}
\newcommand{\twomatrix}[4]{\left(\begin{array}{cc} {#1} & {#2} \\ {#3} & {#4} \end{array} \right)}
\newcommand{\threematrix}[9]{\left(\begin{array}{ccc} {#1} & {#2} & {#3} \\ {#4} & {#5} & {#6}\\ {#7} & {#8} & {#9} \end{array} \right)}
\nc{\Spec}{\operatorname{Spec}}

\nc{\ad}{\operatorname{ad}}  \nc{\Aut}{\operatorname{Aut}}   \nc{\Inn}{\operatorname{Inn}}   \nc{\Lie}{\operatorname{Lie}} \nc{\Ad}{\operatorname{Ad}} \nc{\Der}{\operatorname{Der}} \nc{\rad}{\operatorname{r}} \nc{\kf}{\operatorname{B}}

\nc{\End}{\operatorname{End}} \nc{\rank}{\operatorname{rank}} \nc{\Ker}{\operatorname{Ker}} \nc{\tr}{\operatorname{tr}} \nc{\sym}{\operatorname{sym}} \nc{\diag}{\operatorname{diag}} \nc{\proy}{\operatorname{pr}} \nc{\Adj}{\operatorname{Adj}} \nc{\vspan}{\operatorname{span}}

\nc{\Hess}{\operatorname{Hess}}  \nc{\dif}{\operatorname{d}} \nc{\sen}{\operatorname{sen}} \nc{\grad}{\operatorname{grad}} \nc{\Order}{\operatorname{O}} \nc{\divg}{\operatorname{div}}

\nc{\Iso}{\operatorname{Iso}} \nc{\Diff}{\operatorname{Diff}} \nc{\ricci}{\operatorname{Ric}}  \nc{\Rc}{\operatorname{Rc}} \nc{\Ricci}{\operatorname{Ric}} \nc{\Riem}{\operatorname{Rm}} \nc{\scalar}{\operatorname{sc}} \nc{\scalarm}{\hat{\operatorname{R}}} \nc{\riccim}{\widehat{\operatorname{Ric}}} \nc{\tang}{\operatorname{T}} \nc{\vol}{\operatorname{vol}}

\nc{\mm}{\operatorname{M}} \nc{\CH}{\operatorname{CH}} \nc{\Irr}{\operatorname{Irr}} \nc{\mcc}{\operatorname{mcc}}

\nc{\Id}{\operatorname{Id}}  \nc{\mmm}{\operatorname{m}}

\theoremstyle{plain}
\newtheorem{theorem}{Theorem}[section]
\newtheorem{proposition}[theorem]{Proposition}
\newtheorem{corollary}[theorem]{Corollary}
\newtheorem{lemma}[theorem]{Lemma}
\newtheorem{teointro}{Theorem}
\newtheorem{corointro}[teointro]{Corollary}

\theoremstyle{definition}
\newtheorem{definition}[theorem]{Definition}
\newtheorem{notation}[theorem]{Notation}

\theoremstyle{remark}
\newtheorem{remark}[theorem]{Remark}

\newtheorem{example}[theorem]{Example}

\title[Pluriclosed flow]{The long-time behavior of the homogeneous pluriclosed flow}

\author{Romina M.~Arroyo}
\address{FaMAF $\&$ CIEM, Universidad Nacional de C\'ordoba, C\'ordoba, Argentina}
\email{arroyo@famaf.unc.edu.ar}

\author{Ramiro A.~Lafuente}
\address{Mathematisches Institut, Universit\"at M\"unster, Einsteinstr.~ 62, 48149 M\"unster, Germany}
\email{lafuente@uni-muenster.de}

\thanks{The first author was partially supported by grants from CONICET, FONCYT and SeCyT (Universidad Nacional de C\'ordoba).}

\thanks{The second author was supported by the Alexander von Humboldt Foundation.}

\begin{document}

\begin{abstract}
We study the asymptotic behavior of the pluriclosed flow in the case of left-invariant Hermitian structures on Lie groups. We prove that solutions on $2$-step nilpotent Lie groups and on almost-abelian Lie groups converge, after a suitable normalization, to self-similar solutions of the flow. Given that the spaces are solvmanifolds, an unexpected feature is that some of the limits are shrinking solitons. We also exhibit the first example of a homogeneous manifold on which a geometric flow has some solutions with finite extinction time and some that exist for all positive times.
\end{abstract}

\maketitle

\tableofcontents

\section{Introduction}

The great success of the Ricci flow in Riemannian geometry suggests the idea of looking for a natural counterpart in the realm of Hermitian geometry. Motivated by this, in \cite{ST11} the authors introduced a family of parabolic equations called `Hermitian curvature flows'. In this article we are concerned with one of them, namely the so called \emph{pluriclosed flow}: an evolution equation for a family $\omega(t)$ of Hermitian metrics on a complex manifold $(M^{2n}, J)$ which satisfy the \emph{pluriclosed}  condition $\partial \bar \partial \omega = 0$ (also called \emph{SKT metrics}, short for `Strong K\"ahler with torsion'). The family $\omega(t)$ is called a pluriclosed flow solution if it solves
\[
  \frac{\partial}{\partial t} \omega = - \big( \rho^B(\omega)\big)^{1,1}, \qquad \omega(0) = \omega_0.
\]
Here $ \big( \rho^B(\omega)\big)^{1,1}$  denotes the $(1,1)$-part of the
Ricci form $\rho^B(\omega)$ corresponding to the Bismut connection $\nabla^B$ of $(M^{2n}, J, \omega)$  ---the unique connection making $J$ and $g$ parallel and having totally skew-symmetric torsion---.  Along the article we will use  both $g$ and $\omega(\cdot, \cdot) = g(J \cdot , \cdot)$ to denote a Hermitian metric. If $\omega_0$ is not Kähler, a key point in considering this connection instead of the Levi-Civita connection of $g$ is that the complex structure and the pluriclosed condition are preserved along the flow.

In recent years the pluriclosed flow has been an active subject of study, and already many regularity and convergence results have been established \cite{Str13,Str16}, as well as connections with generalized K\"ahler geometry \cite{ST12}. Remarkably, still not much is known in the homogeneous case. We say a Hermitian manifold $(M^{2n}, J, \omega)$ is \emph{homogeneous}, if there is a real Lie group $\G$ acting transitively on it by isometric biholomorphisms. Previous results in this regard include a complete description of the long-time behavior of the flow on locally homogeneous compact complex surfaces \cite{Boling}, long-time existence on $2$-step nilmanifolds of arbitrary dimension \cite{EFV15}, and on two $6$-dimensional solvmanifolds \cite{FV15}.

It was noticed in \cite{EFV15} that pluriclosed flow solutions on $2$-step nilmanifolds become more and more flat  as $t\to \infty$. It is thus natural to ask what is their asymptotic behavior after an appropriate normalization. In this direction, our first main result is

\begin{teointro}\label{mainthm_nil}
Let $(\N, J)$ be a simply-connected, $2$-step nilpotent Lie group with left-invariant complex structure $J$, and let $(g(t))_{t\in [0,\infty)}$ be a pluriclosed flow solution of left-invariant pluriclosed metrics on $(\N, J)$. Then, the rescaled metrics $(1+t)^{-1} \cdot g(t)$ converge in the Cheeger-Gromov sense  to a non-flat, left-invariant, pluriclosed soliton $(\bar \N, \bar J, \bar g)$, as $t\to\infty$ .
\end{teointro}

A \emph{pluriclosed soliton} is a pluriclosed metric $\bar g$ whose pluriclosed flow evolution is given only by scaling and the action of time-depenedent bi-holomoprhisms --- in other words, a self-similar solution of the flow. The limit group $\bar \N$ in Theorem \ref{mainthm_nil} is again $2$-step nilpotent and simply-connected, but may be non-isomorphic to $\N$. Cheeger-Gromov convergence here means that for any sequence of times there exists a subsequence $(t_k)_{k\in \NN}$ for which the corresponding Hermitian metrics converge in the following sense: there exist maps $\varphi_k : \Omega_k \subset (\bar \N, \bar J) \to (\N, J)$ with $\varphi_k(\bar e) = e$ ($\bar e$ and $e$ being the identity elements in $\bar \N$ and $\N$, respectively) which are bi-holomoprhisms onto its images, defined on open subsets $\Omega_k$ that exhaust $\bar \N$, and such that $\varphi_k^* g(t_k) \to \bar g$ as $k\to\infty$, in $C^\infty$ topology uniformly over compact subsets.

The next question we are interested in is what happens on solvable Lie groups. More precisely, which left-invariant metrics are pluriclosed? And among those, what is the behavior of the pluriclosed flow? We give a complete answer to these questions in the case of \emph{almost-abelian} solvable Lie groups. A Lie group is called almost-abelian if its Lie algebra has a codimension-one abelian ideal. In spite of having a simple Lie-theoretical description, its compact quotients provide a rich family of examples of compact complex manifolds, including for instance hyperelliptic surfaces, Inoue surfaces of type $S^0$, primary Kodaira surfaces (these are also quotients of $2$-step nilpotent Lie groups), and Fino and Tomassini's $6$-dimensional solvmanifold admitting a generalized K\"ahler structure but no K\"ahler structures \cite{FT09}, just to name a few.

To describe our results in this direction we need to introduce first some notation. A left-invariant Hermitian structure $(J,g)$ on a connected Lie group $\G$ is determined by the corresponding tensors on its Lie algebra $\ggo$. If $\G$ is almost-abelian, there is a $g$-orthonormal basis $\bca := \{e_1, \ldots, e_{2n} \}$ for $\ggo$ such that $\{e_1, \ldots, e_{2n-1} \}$ spans the codimension-one abelian ideal $\ngo$. In this way, $(\G, J, g)$ is determined by the real $(2n-1)\times (2n-1)$ matrix corresponding to $\ad(e_{2n})|_{\ngo}$ via $\bca$:
\begin{equation}
  (\ad e_{2n})|_{\ngo} = \twomatrix{a}{0}{v}{A},\qquad a\in \RR, \quad v\in \RR^{2n-2}, \quad A \in \RR^{(2n-2)\times (2n-2)},
\end{equation}
where the blocks correspond to the subspaces spanned by $e_1$ and $ \{e_2,\ldots,e_{2n-1}\}$. Moreover, $A$ commutes with the restriction of $J$ to $\RR^{2n-2}$. This fact and the $0$ in the upper-right block are due to the integrability of $J$ (Lemma \ref{lem_integrmuA}). Accordingly, we denote by $g_{a,v,A}$ (or $\omega_{a,v,A}$) the corresponding left-invariant Hermitian metric on $(\G,J)$. The reader is referred to Section \ref{sec_almostabelian} for further details.

Our second main result characterizes the pluriclosed condition $\partial \bar \partial \, \omega_{a,v,A} = 0$ in terms of the algebraic data $a, v$ and $A$:

\begin{teointro}\label{mainthm_sktaa}
A left-invariant Hermitian structure $(J, g_{a,v,A})$ on an almost-abelian Lie group $\G$ is pluriclosed if and only if $A$ is a normal matrix that commutes with the complex structure, and its eigenvalues have real part equal to $0$ or $-a/2$.
\end{teointro}

Recall that a matrix is called \emph{normal} if it commutes with its transpose. Theorem \ref{mainthm_sktaa} yields examples of pluriclosed metrics on simply-connected solvable Lie groups in arbitrarily high dimensions. It is worthwhile mentioning that many of them admit cocompact lattices, see e.g.~ \cite{FT09}.

Left-invariant pluriclosed Hermitian structures on Lie groups (and its compact quotients) were also studied by other authors. The classification for $4$-dimensional Lie groups was obtained in \cite{MadsenSwann}. In dimension $6$, all nilpotent examples were classified in \cite{FinoPartonSalamon} (see also \cite{Ugarte07}), and some partial results on the solvable case are given in \cite{FinoOtalUgarte}. Regarding higher dimensions, besides some non-existence results proved in \cite{EFV12}, \cite{FinoKasuyaVezzoni15} for nilpotent and solvable Lie groups respectively, to the best of our knowledge no general existence results were previously known.


In our third main result we describe the long-time behavior of the pluriclosed flow on almost-abelian Lie groups. We refer the reader to Theorem \ref{thm_aa} for a more precise statement.

\begin{teointro}\label{mainthm_aa}
Let $(g(t))_{t\in[0,T)}$ be a maximal pluriclosed flow solution of left-invariant pluriclosed metrics on a simply-connected, almost-abelian Lie group with left-invariant complex structure $(\G,J)$. If $\G$ is unimodular, then the flow exists for all positive times ($T=\infty$). On the other hand, for non-unimodular $\G$ there exist examples with finite extinction time ($T<\infty$). In any case, suitably normalized solutions converge in the Cheeger-Gromov sense to pluriclosed solitons.
\end{teointro}

A connected Lie group $\G$ is called \emph{unimodular} if all the adjoint maps of its Lie algebra are traceless. Let us mention here that the limit solitons in Theorems \ref{mainthm_nil} and \ref{mainthm_aa} are in fact \emph{algebraic solitons}, that is, the bi-holomporhisms giving the corresponding flow evolution are also Lie group automorphisms, see Definition \ref{def_soliton} and Proposition \ref{prop_solitonevol}. Another interesting feature is that for  $\G$ unimodular, the limit Hermitian structure is compatible with a \emph{generalized K\"ahler structure} (see Section \ref{sec_SKTaa} for the definition).

\begin{corointro}
There exist a simply-connected, almost-abelian Lie group $(\G, J)$ with left-invariant complex structure $J$, on which the pluriclosed flow of left-invariant metrics has some solutions with finite extinction time and some that exist for all positive times.
\end{corointro}

As far as we know, this is the first example of a geometric flow on a homogeneous space exhibiting such an unexpected behavior. A concrete example is given in Example \ref{exa_shrinksteady}.

The simply-connectedness assumption in all our main results is justified by two reasons. On one hand, the local geometry of a pluriclosed flow evolution of left-invariant structures on a compact solvmanifold  is completely determined by the behavior of the corresponding pull-back solution on the universal cover, as they are locally isomorphic. On the other hand, the assumption has proved to be extremely useful for realizing the potential limits under smooth convergence, by avoiding collapsing situations (cf.~\cite{Lot07} for the Ricci flow case).




We turn now to the proofs of our main results. Theorems \ref{mainthm_nil} and  \ref{mainthm_aa} are obtained using Lauret's \emph{bracket flow} approach: the  flow of left-invariant pluriclosed structures is equivalent (up to pull-back by bi-holomorphisms) to an ODE on the variety of Lie algebras. The latter is viewed as an algebraic subset of the vector space $V_{2n} := \Lambda^2 (\RR^{2n})^* \otimes \RR^{2n}$. Resembling the Ricci flow (see \cite{nilRF}), in Theorem \ref{mainthm_nil} the ODE analysis is made possible by a close link with Geometric Invariant Theory applied to the natural $\Gl_{2n}(\RR)$-action on $V_{2n}$: the bracket flow coincides with the negative gradient flow of the norm squared of the moment map corresponding to a certain group action on $V_{2n}$. For Theorem \ref{mainthm_aa}, we follow the ideas in \cite[$\S$3]{BL17} and find an appropriate \emph{gauge} for the bracket flow so that the equation can be written as an ODE for the data $a, v$ and $A$. This ODE simplifies significantly after a suitable normalization. Finally, Theorem \ref{mainthm_sktaa} follows by expressing the pluriclosed condition in terms of $a$ and $A$ (see \eqref{eqn_SKTcondition}) and then applying some linear algebra estimates relating the norm of $A$ and the real part of its eigenvalues.

The organisation of the article is as follows. In Section \ref{sec_prelim} we give the necessary preliminaries on pluriclosed metrics, the bracket flow approach with gauging, and pluriclosed solitons. Section~ \ref{sec_nilmfd} is devoted to the proof of Theorem \ref{mainthm_nil}. Theorems \ref{mainthm_sktaa} and \ref{mainthm_aa} are proved in Section \ref{sec_almostabelian}, and a summarized form of the different types of behaviors of the flow in this case is presented in Table~ \ref{tab_almostabelian}. Finally, Appendix \ref{app} contains some basic linear algebra estimates for which we could not find a reference.


\section{Preliminaries} \label{sec_prelim}

Let $(M^{2n},J,g)$ be a Hermitian manifold with complex structure $J$ and compatible Riemannian metric $g$, and denote by $\omega(\cdot,\cdot)=g(J \cdot,\cdot)$ its fundamental $2$-form. The Bismut connection $\nabla^B$ on $M$ is the unique Hermitian connection ($J$ and $g$ are parallel) with totally skew-symmetric torsion, that is, the tensor $c(X,Y,Z) := g(X, T^B(Y,Z))$ is a $3$-form, where $T^B(Y,Z) = \nabla^B _Y Z - \nabla^B_Z Y - [Y,Z]$ is the torsion of $\nabla^B$ (see \cite{Bis89,Gau97}).

We say $\omega$ (or $g$) is \emph{pluriclosed}, also \emph{strong K\"ahler with torsion} ---SKT for short---, if $\partial \bar \partial \omega =~ 0$. This condition is equivalent to the torsion $3$-form $c$ being closed.  Indeed,  $c = -d^c \omega =  d \omega(J \cdot , J\cdot , J\cdot)$ and $d d^c = 2 \sqrt{-1} \partial \overline \partial$, where $d^c = \sqrt{-1}(\overline \partial - \partial)$ is the real Dolbeault operator.

The \emph{pluriclosed flow}, introduced in \cite{ST10}, is the parabolic flow for SKT metrics defined by
\begin{equation}\label{eqn_pluriflow}
  \frac{\partial}{\partial t} \omega = - \left( \rho^B\right)^{1,1}, \qquad \omega(0) = \omega_0.
\end{equation}
Here, $(\rho^B)^{1,1}(\cdot, \cdot) = \unm( \rho^B(\cdot, \cdot) + \rho^B(J \cdot , J \cdot) )$ denotes the $(1,1)$-part of the Bismut-Ricci form $\rho^B = \rho^B(\omega)$, defined by
\begin{equation}\label{eqn_defBismutRic}
  \rho^B(X,Y)=-\sum _{k=1}^n g(R^B(X,Y)e_k,J e_k),
\end{equation} where $R^B(X,Y) = [\nabla^B_X,\nabla^B_Y] - \nabla^B_{[X,Y]}$ is the curvature tensor of $\nabla^B$ and $\{e_k, Je_k\}$ is a local $g$-orthonormal real frame.
It is an analog of the K\"ahler-Ricci flow in the non-K\"ahler setting, and it preserves the SKT condition. Moreover, if the initial condition $\omega_0$ is K\"ahler both flows coincide, see \cite{ST10}.

\subsection{SKT metrics on Lie groups}

We call a Hermitian manifold $(M^{2n},J,g)$ \emph{left-invariant} if its universal cover $\tilde M$ is diffeomorphic to a simply-connected Lie group $\G$, and $\pi^* J$, $\pi^* g$ are left-invariant tensors defining a Hermitian structure on $\G$, where $\pi : \G \to M$ denotes the covering map. This is for instance the case for invariant Hermitian structures on $M = \Gamma \backslash \G$, where $\Gamma$ is a cocompact discrete subgroup of $\G$.

If we denote the Lie algebra of $\G$ by $\ggo$, then according to \cite[(3.2)]{EFV12} the torsion $3$-form of the Bismut connection of a left-invariant Hermitian manifold can be computed by
\begin{equation}\label{eqn_formulac}
  c(X,Y,Z) = -g([JX, JY], Z) - g([JY,JZ], X) -g([JZ,JX], Y), \qquad X, Y, Z \in \ggo.
\end{equation}
Its exterior derivative is thus given by
\begin{align}\label{eqn_formuladc}
  dc(W,X,Y,Z) =& -c([W,X],Y,Z) + c([W,Y],X,Z) -c([W,Z],X,Y) \\
                &- c([X,Y],W,Z) + c([X,Z],W,Y) - c([Y,Z],W,X). \nonumber
\end{align}
The SKT condition $dc=0$ can thus be written as a system of equations on $\ggo$ involving the Lie bracket and the complex structure.

\subsection{The bracket flow approach to the homogeneous pluriclosed flow}\label{section_BF}

We now briefly introduce the principle of varying brackets instead of geometric structures, which has been applied extensively and quite successfully for studying locally homogeneous geometric structures. In the case of (almost) Hermitian structures, this was developed in detail by Lauret in \cite{Lau15}. See also \cite{LW17}.

Let $(\G, J, g)$ be a $2n$-dimensional simply-connected Lie group with a left-invariant Hermitian structure, and denote its Lie algebra by $\ggo$. This object is uniquely determined by the infinitesimal  data
\[
  J : \ggo \to \ggo,  \qquad \ip : \ggo \times \ggo \to \RR, \qquad \mu : \Lambda^2 \ggo \to \ggo,
\]
where $J$ is a linear endomorphism with $J^2 = -\Id_\ggo$, $\ip$ is a scalar product, $\la J \cdot , J \cdot\ra = \ip$, and $\mu$ denotes the Lie bracket, regarded as an element of the vector space
\[
  \Vg := \Lambda^2\ggo^* \otimes \ggo
\]
which satisfies in addition the Jacobi identity.

In order to emphasize the Lie bracket we will often consider $\ggo$ just as the underlying real vector space, and denote a Lie algebra structure by $(\ggo,\mu)$ or simply $\mu$. In this case, $\G_\mu$ will denote  the simply-connected Lie group  with Lie algebra $(\ggo,\mu)$, and $(J_\mu, g_\mu)$ the left-invariant Hermitian structure on $\G_\mu$ determined by $(J,\ip)$ on $(\ggo,\mu) \simeq T_e \G_\mu$.

The idea of the `varying brackets principle' is as follows: one fixes once and for all $J$, $\ip$ on~ $\ggo$. Any $J$-compatible left-invariant metric $g'$ on $\G$ is determined by a scalar product on $\ggo$ of the form $\ip' = \la h \, \cdot , h \, \cdot\ra$, with $h$ an element of
\[
    \Gl(\ggo, J)  := \{ \phi \in \Gl(\ggo) : [\phi ,J] = 0 \}\simeq \Gl_n(\CC).
\]
In other words, $(\G, J, g')$ is determined by the data $\big(\ggo,\mu, J, \ip' = \la h\cdot, h \cdot \ra \big)$. The linear map
\[
    h : \big(\ggo,\mu, J, \la h\cdot, h \cdot \ra \big) \to \big(\ggo,h\cdot \mu, J, \ip \big)
\]
is at the same time orthogonal, holomorphic, and a Lie algebra isomorphism. Here,
\begin{equation}\label{eqn_actionbrackets}
  (h\cdot \mu)(\cdot, \cdot ) := h \mu(h^{-1} \cdot , h^{-1} \cdot ), \qquad h\in \Gl(\ggo), \qquad \mu \in \Vg,
\end{equation}
denotes the linear `change of basis' action of $\Gl(\ggo)$ on $\Vg$. Finally, the Lie group isomorphism $\varphi : \G = \G_\mu \to \G_{h\cdot \mu}$ determined by $d \varphi |_e = h$ is an equivalence between the Hermitian structures
\[
    \varphi : (\G, J, g')  \to (\G_{h\cdot \mu}, J_{h\cdot \mu}, g_{h\cdot \mu})
\]
defined respectively by the data $\big(\ggo,\mu, J, \la h\cdot, h \cdot \ra \big)$, $\big(\ggo,h\cdot \mu, J, \ip \big)$. Moreover, $\varphi$ is equivariant with respect to the transitive actions of $\G$, $\G_{h\cdot \mu}$ by left multiplication.
 This yields

\begin{proposition}\label{prop_movingbrackets}
Let $\G$ be a simply-connected Lie group with Lie algebra $(\ggo, \mu)$. After fixing $J$, $\ip$ on $\ggo$, the space of left-invariant Hermitian structures on $\G$ may be parameterized by the orbit $\Gl(\ggo, J) \cdot \mu$ in the space of brackets $\Vg$.
\end{proposition}


\begin{remark}\label{rmk_gauge}
Under the correspondence described in Proposition \ref{prop_movingbrackets}, two brackets $\mu_1$, $\mu_2$ satisfying $\mu_1 = k \cdot \mu_2$, with $k$ in the unitary group
\[
   \U(\ggo,J) := \{ h\in \Gl(\ggo, J) : h^{-1} = h^t  \},
\]
give rise to the \emph{same} Hermitian manifold. Indeed, if $\mu_2 = h_2\cdot \mu$ for $h_2 \in \Gl(\ggo,J)$, then $\mu_1 = k h_2 \cdot \mu$, and the corresponding scalar products $\ip_1 = \la k h_2 \cdot, k h_2 \cdot \ra$ and $\ip_2 = \la h_2 \cdot, h_2 \cdot\ra$ on $\ggo$ coincide.
\end{remark}

With regards to pluriclosed metrics, it was shown in \cite[$\S$4]{EFV15} that the SKT condition translated into the brackets setting is a set of polynomial equations in $\Vg$. Thus, the set of brackets giving rise to SKT metrics is a real algebraic subset of the orbit $\Gl(\ggo,J) \cdot \mu$.

It is natural to ask how the pluriclosed flow  of left-invariant SKT structures looks like in the space of brackets. The first answer to this question is given by the so called \emph{bracket flow},
\begin{equation}\label{eqn_pluriBF}
    \ddt \mu = -\pi(P_\mu) \mu, \qquad \mu(0) = \mu_0,
\end{equation}
an ODE for a curve $\mu = \mu(t)$ in $\Vg$, where $P_\mu \in \End(\ggo)$ are implicitly defined by
\begin{equation}\label{eqn_defPmu}
  \omega(P_\mu \, \cdot \, , \cdot )  = \unm  \big(\rho_\mu^B \big)^{1,1} (\cdot, \cdot),
\end{equation}
$\rho^B_\mu \in \Lambda^2 \ggo^*$ here denotes the Bismut-Ricci form of the Hermitian manifold associated with $\mu$ and $\pi$ is the representation induced by the action (\ref{eqn_actionbrackets}) defined by
\begin{equation}
(\pi(A)\mu)(X,Y) = A\mu(X,Y)-\mu(AX,Y)-\mu(X,AY), \quad A\in \Gl(\ggo), \quad X,Y \in \ggo.
\end{equation}
It follows from \cite[Thm.~ 4.2]{EFV15}, \cite[Thm.~ 5.1]{Lau15} that the unique solution $\mu(t)$ to \eqref{eqn_pluriBF} defines ---as explained above--- a family of Hermitian manifolds that coincides, up to pull-back by time-dependent biholomorphisms, with the pluriclosed flow starting at the corresponding SKT metric. Following \cite[$\S$3]{BL17} we will prove a slightly more general result, which will allow us to write down the bracket flow equations using fewer variables. This will be of great help in our study of the flow on almost-abelian solvmanifolds.  Recall that
\[
  \ug(\ggo, J) := \{ A\in \glg(\ggo, J) : A = -A^t \} \simeq \ug(n)
\]
is the Lie algebra of the unitary group $\U(\ggo, J) \simeq \U(n)$.

\begin{theorem}\label{thm_gaugedBF}
Let $(\G, J, g_0)$ be a left-invariant SKT structure on the simply-connected Lie group $\G$, with associated bracket $\mu_0 \in \Vg$. Let $\omega(t)(\cdot,\cdot) = g(t) (J \cdot, \cdot)$ and $\bar \mu(t)$ denote respectively the solutions to the pluriclosed flow \eqref{eqn_pluriflow} and to the \emph{gauged bracket flow}
\begin{equation}\label{eqn_gaugedBF}
    \ddt \bar \mu = -\pi(P_{\bar \mu} - U_{\bar \mu}) \bar \mu, \qquad \bar \mu(0) = \mu_0,
\end{equation}
where $\Vg \ni \bar \mu \mapsto U_{\bar \mu} \in \ug(\ggo, J)$ is an arbitrary smooth map. Then, both solutions exist for the same interval of time $I\subset \RR$, and if $(\G_{\bar \mu}, J_{\bar \mu}, g_{\bar \mu})$ denotes the Hermitian manifold associated with $\bar \mu$ (defined as above, after fixing $J := J(e)$, $\ip := g_0(e)$ on $\ggo$), then there exist time-dependent, equivariant diffeomorphisms
\[
    \varphi_t : (\G, J, g(t))  \, \to    \,   \big(\G_{\bar \mu(t)}, J_{\bar \mu(t)}, g_{\bar \mu(t)} \big),
\]
such that $\varphi_t^* J_{\bar \mu(t)} = J$ and $\varphi_t^* g_{\bar \mu(t)} = g(t)$ for all $t \in I$.
\end{theorem}

\begin{proof}
The result will follow from \cite[Thm.~ 5.1]{Lau15} and Remark \ref{rmk_gauge} once we show that $\bar \mu(t) = k(t) \cdot \mu(t)$  for some one-parameter family $\{k(t) \}_{t\in I} \subset \U(\ggo, J)$, where $\mu(t)$ is the solution to the bracket flow \eqref{eqn_pluriBF}. To do this, as in the proof of \cite[Prop.~ 3.1]{BL17} we let $k(t)$ solve the ODE
\[
  k'(t) = U_{k(t) \cdot \mu(t)} \, k(t), \qquad k(0) = \Id_\ggo.
\]
Since $k' k^{-1}$ always belongs to the Lie algebra of $\U(\ggo,J)$, a standard ODE argument shows that one has $k(t) \in \U(\ggo,J)$ for all $t$ for which the bracket flow $\mu(t)$ is defined. After setting $\tilde \mu(t) := k(t) \cdot \mu(t)$ and using linearity of the action one obtains
\[
    \tilde \mu' = -\pi\big(k \, P_{\mu}  \, k^{-1} - U_{k\cdot \mu} \big) (k \cdot \mu) = - \pi(P_{\tilde \mu} - U_{\tilde \mu}) \, \tilde \mu.
\]
The second equality follows by the $\U(\ggo,J)$-equivariance of the map $\mu \mapsto P_\mu$, which can be checked by a routine computation. From uniqueness of ODE solutions it follows that $\bar \mu(t) = \tilde \mu(t)$, as we wanted to prove.
\end{proof}

\begin{remark}\label{rmk_ht}
It follows from \cite[Thm.~5.1]{Lau15} and the proof of Theorem \ref{thm_gaugedBF} that the solution $\bar \mu(t)$ to \eqref{eqn_gaugedBF} may be written as  $\bar \mu(t) = h(t) \cdot \mu_0$, where $\{h(t) \}\subset \Gl(\ggo)$ solves
\[
      h' = -\left(P_{\bar \mu(t)} - U_{\bar \mu(t)} \right) h, \qquad h(0) = \Id_\ggo.
\]
\end{remark}

In order to understand the asymptotic behavior of a geometric flow it is important to study suitably normalized solutions. To that end, recall that the scalar product on $\ggo$ induces naturally a scalar product on $\Vg$, see e.g.~\cite[(20)]{homRF}. Following the ideas in \cite{homRF}, we explain how the normalized solutions can be regarded as solutions to a similar equation:

\begin{lemma}\label{lem_BFnormeqn}
If $\bar \mu(t)$ is a solution to \eqref{eqn_gaugedBF} then, up to time reparameterization, the family $\nu(t) := \bar \mu(t)/\Vert \bar \mu(t) \Vert$ is a solution to the \emph{normalized gauged bracket flow}
\begin{equation}\label{eqn_BFnorm}
  \ddt \nu = -\pi (P_\nu - U_\nu + r_\nu \Id_\ggo) \nu,
\end{equation}
where $r_\nu :=  \la  \pi (P_\nu - U_\nu) \nu, \nu\ra$.
\end{lemma}

\begin{proof}
Since $\pi(\Id_\ggo) \nu = - \nu$, the right-hand-side in \eqref{eqn_BFnorm} is nothing but $X(\nu) + r_\nu \, \nu$, where $X(\nu)$ is the vector field on $V(\ggo)$ defining the gauged bracket flow ODE. Observe that $r_\nu \, \nu$ is minus the projection of $X(\nu)$ onto the line spanned by $\nu$, thus $X(\nu) + r_\nu \, \nu$ is tangent to the unit sphere in $V(\ggo)$. Hence its integral curves are nothing but the projection of the integral curves of $X$ onto the sphere, and the lemma follows. We refer the reader to \cite[$\S$3.3]{homRF} for a more detailed explanation.
\end{proof}

 \begin{remark}\label{rem_BFnorm}
It is also possible to consider other normalizations of the bracket flow. They are all of the form $c(t)  \, \bar \mu(t)$, for some real valued function $c(t) > 0$. It follows exactly as in \cite[$\S$3.3]{homRF} that, up to a time reparameterization, they solve the ODE \eqref{eqn_BFnorm} for a suitably chosen $r_\nu$.
 \end{remark}

\subsection{Static SKT metrics and pluriclosed solitons}\label{section_solitons}

In \cite{ST10} the authors define the notion of a \emph{static SKT metric}: a Hermitian metric $g$ on a complex manifold $(M^{2n}, J)$ is called static if its Bismut-Ricci form satisfies
\begin{equation}\label{eqn_static}
    (\rho^B)^{1,1} = \alpha \, \omega, \qquad \alpha \in \RR.
\end{equation}
Static metrics are to the pluriclosed flow what Einstein metrics are to the Ricci flow. When $\alpha\neq 0$, \eqref{eqn_static} implies the existence of a so called \emph{Hermitian-symplectic} structure, see \cite{ST10}.

In the homogeneous case these metrics were studied in \cite{Enr13}, where it is shown that there are no examples on 2-step nilmanifolds other than torii. In Section \ref{sec_almostabelian} we prove that the same is true on almost-abelian solvmanifolds when $\alpha \neq 0$, see Remark \ref{rem_staticaa}. On the other hand, for $\alpha = 0$ there exist some new examples: see Corollary \ref{cor_plurisolitons} and Remark  \ref{rem_steadystatic}.

The non-existence of invariant static metrics in these spaces suggests the idea of enlarging the class of ``special solutions'' to the pluriclosed flow. A natural way of saying that a structure is ``special'' with respect to a geometric flow is that its corresponding evolution is \emph{self-similar}, meaning that the flow starting at said structure evolves only by scaling and pull-back by diffeomorphisms preserving the underlying geometric structure (in our case the complex structure~ $J$). Moreover, in the presence of symmetries one may even require the latter to be equivariant under the group action. In this direction and following \cite{Lau15} we define

\begin{definition}\label{def_soliton}
A Lie group endowed with a left-invariant SKT structure $(\G, J, g)$ is an \emph{algebraic pluriclosed soliton} if the endomorphism $P\in \End(\ggo)$ defined by $\omega(P \cdot, \cdot) = \unm (\rho^B)^{1,1}(\cdot,\cdot)$ satisfies
\begin{equation}\label{eqn_defsoliton}
P = \alpha \Id_\ggo + \unm \left( D + D^t \right),  \qquad \alpha \in \RR, \quad D \in \Der(\ggo), \quad [D,J] =0.
\end{equation}
\end{definition}

\begin{remark}\label{rem_algebraic}
In the Ricci flow context, the analogous to \eqref{eqn_defsoliton} is called a \emph{semi-algebraic soliton}, and they are called algebraic when $D = D^t$, cf.~also \cite[Def.~4.11]{Lau15b}. Since in the case of the pluriclosed flow these notions are not equivalent (see Remark \ref{rem_semialg}), we decided to save the word algebraic for \eqref{eqn_defsoliton} because it seems more natural, mainly due to Proposition \ref{prop_solitonevol} below.
\end{remark}

In connection with the previous section we have the following useful characterization:

\begin{proposition}\label{prop_solitonfixedpoint}
A left-invariant SKT structure $(\G,J,g)$ is an algebraic pluriclosed soliton if and only if it is a fixed point of a normalized bracket flow \eqref{eqn_BFnorm}, for some choice of gauge.
\end{proposition}

\begin{proof}
If the  bracket $\nu$ corresponding to $g$ is a fixed point of \eqref{eqn_BFnorm} for some gauging maps $U_\nu$, then $P_\nu -U_\nu + r_\nu \Id_\ggo =: D \in \Der(\ggo,\nu)$. Using that $U_\nu$ is skew-symmetric and $P_\nu + r_\nu \Id_\ggo$ symmetric, we easily obtain \eqref{eqn_defsoliton}. Recall that both $U_\nu$ and $P_\nu$ commute with $J$.

Conversely, if $\nu$ satisfies \eqref{eqn_defsoliton} then one chooses $U_\nu :=- \unm (D-D^t) \in \ug(\ggo,J)$ as gauging map, and obtains a fixed point of the corresponding normalized flow \eqref{eqn_BFnorm}.
\end{proof}

As mentioned above, the relevance of these special SKT structures is that they give rise to self-similar solutions to the pluriclosed flow:

\begin{proposition}\label{prop_solitonevol}
If $(\G, J, \omega_0)$ is an algebraic pluriclosed soliton then the solution $\omega(t)$ to the pluriclosed flow \eqref{eqn_pluriflow} is given by $\omega(t) = c(t) \, \varphi_t^* \omega_0$, where $c(t) \in \RR$ and $\varphi_t$ is a one-parameter family of automorphisms of the complex manifold $(\G,J)$ which are also Lie automorphisms of~ $\G$.
\end{proposition}

\begin{proof}
Let us fix $J:= J(e)$, $\ip:= g_0(e)$ on $\ggo$, and denote by $\mu_0 \in \Vg$ the bracket corresponding to $g_0$. We have
\[
  P_{\mu_0} = \alpha \Id_\ggo + \unm (D + D^t), \qquad \alpha\in \RR, \qquad  D\in \Der(\ggo, \mu_0) \cap \glg(\ggo,J).
\]
We claim that the solution of the gauged bracket flow equation \eqref{eqn_gaugedBF} is given simply by scaling $\bar \mu(t) = c(t) \mu_0$, for suitable scalars $c(t) \in \RR$ and a suitable gauging $\{U_{t}\} \subset \ug(\ggo,J)$. Indeed, let $c(t) = (1-2\alpha \, t)^{-1/2}$ be the solution to
\[
    c' = \alpha \, c^3, \qquad c(0) = 1,
\]
and consider
\[
  U_t := -\unm c(t)^2  (D-D^t) \in \ug(\ggo,J).
\]
Using that $P_{c\cdot \bar \mu} = c^2 P_{\bar \mu}$ and $\pi(D)\mu_0 = 0$ we  obtain
\[
    - \pi (P_{\bar \mu(t)} - U_t) \bar \mu(t) = - c(t)^3 \, \pi(P_{\mu_0} - U_0) \mu_0 = -c(t)^3 \, \pi(\alpha \Id_\ggo + D) \mu_0 =  \alpha \, c(t)^3 \mu_0 = (\bar \mu(t))',
\]
from which our claim follows. The proposition now also follows by Theorem \ref{thm_gaugedBF}. Indeed, in the notation of that theorem, up to scaling we have that $(\G_{\bar \mu(t)}, J_{\bar\mu(t)}, g_{\bar \mu(t)})$ is constant equal to $(\G, J, g_0)$. Thus, after modifying them with the scaling factors, the equivariant diffeomorphisms $\varphi_t$ given by Theorem \ref{thm_gaugedBF} become Lie automorphisms of $\G$, and also automorphisms of the complex manifold $(\G,J)$, as they preserve $J$.
\end{proof}

\section{Nilmanifolds}\label{sec_nilmfd}

In this section we study the pluriclosed flow of invariant SKT structures on a simply-connected, non-abelian, nilpotent Lie group $(\N,J)$ endowed with a left-invariant complex structure $J$.

\subsection{SKT nilmanifolds}

All known examples of nilpotent Lie groups admitting a left-invariant SKT structure are $2$-step nilpotent. The latter were thought to exhaust all nilpotent examples, see \cite[Thm.~1.2]{EFV12}. However, it was pointed out to us by the authors of \cite{EFV12} in a personal communication that there is a gap in the proof of the above mentioned theorem. In particular, to the best of our knowledge there could a priori exist examples in $k$-step nilpotent Lie groups with $k>2$. A quick inspection of the proofs shows that \cite[Prop.~3.1]{EFV12} is still valid, and so are the results in \cite{EFV15} provided one adds the $2$-step nilpotent assumption.

By \cite[Prop.~3.1]{EFV12} the center  $\zg$ of the Lie algebra $\ngo$ of a $2$-step nilpotent Lie group with SKT structure is preserved by $J$.
Thus, let us fix $(\ngo, J)$ a $2$-step nilpotent Lie algebra with a complex structure $J$ such that $J(\zg) \subset \zg$. Let $g$ be a $J$-compatible metric on $(\N, J)$, determined by the scalar product $\ip$ on $\ngo$, and denote by $\vg$ the orthogonal complement of $\zg$ in $\ngo$. The 2-step condition may be rewritten as
\[
  \ngo = \vg \oplus \zg, \qquad [\ngo,\ngo] = [\vg,\vg] \subseteq \zg.
\]

\begin{lemma}\label{lem_SKTonlyzg}
The fact that $g$ is SKT depends only on the restriction of $g$ to $\zg$.
\end{lemma}

\begin{proof}
Notice that in  formula \eqref{eqn_formulac} for the torsion $3$-form $c$, if one of the entries, say $Z$, lies in the center $\zg$, then one gets
\begin{equation}\label{eqn_cXcenter}
  c(X,Y,Z) = -\big\la  [JX,JY], Z  \big\ra, \qquad Z\in \zg, \quad X, Y\in \ngo,
\end{equation}
whereas if two of them are in $\zg$, then $c$ vanishes. Since $[\ngo,\ngo]\subseteq \zg$, it follows from \eqref{eqn_formuladc} that
\[
  dc(\cdot,\cdot,\cdot,Z) = 0, \qquad  \forall \, \,Z\in \zg.
\]
Thus, $dc=0$ if and only if $dc(W,X,Y,Z) =0$, $\forall W,X,Y,Z \in \vg$. Using \eqref{eqn_cXcenter} one gets
\begin{align*}
  dc(W,X,Y,Z) = & \,\, \big\la [W,X],[JY,JZ] \big\ra - \big\la[W,Y],[JX,JZ]\big\ra + \big\la[W,Z],[JX,JY]\big\ra \\
                & + \big\la[X,Y],[JW,JZ]\big\ra - \big\la[X,Z],[JW,JY]\big\ra + \big\la[Y,Z],[JW,JX]\big\ra,
\end{align*}
from which the lemma follows.
\end{proof}

Notice that for a given SKT metric $g$ on $(\N, J)$, by the previous lemma there is a family of SKT metrics on $(\N,J)$ naturally associated to $g$, determined by the scalar products $ \big\{ \la h \, \cdot, h \, \cdot  \ra \,\, : \,\,  h \in  \Gl(\vg,J) \big\}$ on $\ngo$.

Here the group $\Gl(\vg,J)$ is given by
\begin{equation}\label{eqn_defglvJ}
  \Gl(\vg,J) :=   \left\{ h\in \Gl(\vg) : h \, J|_\vg = J|_\vg \, h  \right\} \simeq \Gl_{\dim \vg / 2}(\CC),
\end{equation}
considered as a subgroup of $\Gl(\ngo)$ via the embedding $\Gl(\vg) \subset \Gl(\ngo)$, $h\mapsto  \minimatrix{h}{0}{0}{\Id}$, where the blocks are in terms of $\ngo = \vg \oplus \zg$.

\subsection{Long-time behavior of the pluriclosed flow on nilmanifolds}

Let $(\N,J, g)$ be a left-invariant SKT structure with associated bracket $\mu \in V(\ngo) := \Lambda^2(\ngo^*) \otimes \ngo$. Consider the $\ip$-orthogonal decomposition $\ngo = \vg \oplus \zg$, where $\zg$ is the center of $(\ngo,\mu)$. Recall that by \cite{EFV15} (see also \cite{PV17}) the endomorphism $P_\mu$  defined in \eqref{eqn_defPmu} is given by
\begin{equation}\label{eqn_Pmunil}
  P_\mu = \twomatrix{\left(\Ricci_\mu\right)_\vg^{1,1}}{0}{0}{0},
\end{equation}
where $(A)_\vg := \proy_\vg \circ A \, |_\vg$ denotes the orthogonal projection $\End(\vg\oplus \zg) \to \End(\vg)$, and $\Ricci_\mu$ is the Ricci endomorphism of $(\ngo, \mu, g)$, that is, the Ricci curvature endomorphism of the corresponding left-invariant Riemannian metric $g$ on the nilpotent Lie group $\N_\mu$ whose Lie algebra is given by $(\ngo,\mu)$ (see \cite[Lemma 4.1]{nilRF}).

Following the discussion of Section \ref{section_BF}, we study the solutions to the ODE on $V(\ngo)$ given by
\begin{equation}\label{eqn_BFnil}
  \ddt \mu = -\pi \left( P_\mu\right) \mu, \qquad \mu(0) = \mu_0.
\end{equation}
Let $\mu(t)$ be the unique solution to \eqref{eqn_BFnil}. We first explain why the center is preserved.

\begin{lemma}\cite{EFV15}\label{lem_centerpreserved}
The center $\zg$  of $(\ngo,\mu_0)$ is also the center of $(\ngo,\mu(t))$, for all $t$.
\end{lemma}
\begin{proof}
By \cite[Thm. 5.1]{Lau15} the solution to the bracket flow can be expressed as $\mu(t) = h(t) \cdot \mu_0$, where $h(t)$ solves
\[
    \ddt h = - P_\mu \, h , \qquad h(0) = \Id.
\]
(Recall that here $Q_\mu^{ac} = 0$ in the notation of that theorem, by integrability.) Notice that if $h = \minimatrix{h_\vg}{0}{0}{\Id}$ then $P_{h\cdot \mu_0}$ has the form \eqref{eqn_Pmunil} with respect to the fixed decomposition $\ngo = \vg \oplus \zg$. Thus, by existence and uniqueness of ODE solutions it is clear that $h(t)$ must be in block form as above, and hence the center of $\mu(t)$ is $\zg$ for all $t$.
\end{proof}

From Lemma \ref{lem_SKTonlyzg} and the proof of Lemma \ref{lem_centerpreserved}, it follows that ---as expected--- the SKT condition is preserved by the flow. Indeed, the corresponding family of metrics $g(t) = g_0 (h(t) \cdot, h(t) \cdot)$ vary only on $\vg$.

It was shown in \cite{EFV15} (see also \cite{PV17}) that pluriclosed flow solutions consisting of left-invariant SKT metrics on $2$-step nilpotent Lie groups are \emph{immortal}, i.e.~defined for all $t\in [0,\infty)$. The main result of this section is the following

\begin{theorem}\label{thm_limitnil}
Let $(\ngo,\mu_0)$ be a $2$-step nilpotent Lie algebra, and let $\mu(t)$ denote the solution to the pluriclosed bracket flow \eqref{eqn_BFnil}. Then, $\mu(t)/\Vert \mu(t) \Vert$ converges as $t\to\infty$ to a non-flat algebraic pluriclosed soliton $\mu_\infty$.
\end{theorem}

In what remains of this section we work towards a proof of this theorem. To that end, we explain a connection between the Bismut-Ricci form of left-invariant SKT metrics on $2$-step nilpotent Lie groups and real geometric invariant theory. Even though all what is needed for proving Theorem \ref{thm_limitnil} is Corollary \ref{cor_BFgradf}, which may be stated without any GIT terminology, we believe it is still worthwhile to explain these links that inspired our results.

Consider the group $\Gl(\vg, J)$ described in \eqref{eqn_defglvJ}, and denote its Lie algebra by $\glg(\vg,J)$.
If $\Or(\ngo) = \Or(\ngo,\ip)$, $\pg = \sym(\ngo,\ip)$ are the symmetric endomorphisms, and $\exp$ is the Lie exponential map, then $\Gl(\ngo) = \Or(\ngo) \exp( \pg)$  is a  Cartan decomposition, inducing the following Cartan decomposition on  $\Gl(\vg,J)$:
\[
  \Gl(\vg, J) = \Or(\vg,J) \exp(\pg({\vg,J})).
\]
Here $\Or(\vg, J) := \Or(\ngo) \cap \Gl(\vg, J)$ and $\pg(\vg,J) = \pg \cap \glg(\vg,J)$. The inner product $\ip$ on $\ngo$ induces naturally inner products on $V(\ngo)$, $\glg(\ngo)$ and $\glg(\vg, J) \subset \glg(\ngo)$, which we also denote by $\ip$.

Recall that for $\mu\in V(\ngo)$ a non-abelian nilpotent Lie bracket, it was shown in \cite[Rmk.~4.9]{Lau2003} that the \emph{moment map} (see \eqref{eqn_defmm} below for its definition) for the $\Gl(\ngo)$-action on $V(\ngo)$ \eqref{eqn_actionbrackets} is related to the Ricci curvature of the corresponding Riemannian metric by the formula
\begin{equation}\label{eqn_Ricnil}
  \mmm_{\Gl(\ngo)} (\mu) =  \frac{4}{\Vert\mu\Vert^2} \cdot \Ricci_\mu.
\end{equation}
Notice that the $2$-step assumption implies in particular that $\mu$ is non-abelian, hence \eqref{eqn_Ricnil} applies.

\begin{lemma}\label{lem_nilPmumm}
If $\mu \in V(\ngo)$ is a $2$-step nilpotent Lie bracket with center $\zg$ and $\vg := \zg^\perp \subset \ngo$, then the moment map for the action of $\Gl(\vg,J)$ on $V(\ngo)$ is given by
\[
  \mmm_{\Gl(\vg,J)} (\mu) = \frac{4}{\Vert \mu \Vert^2} \cdot P_\mu.
\]
In particular, $\tr P_\mu = -\tfrac{\Vert \mu \Vert^2}{2}$. 
\end{lemma}

\begin{proof}
By definition of moment map for the action of a real reductive Lie group $\G = \K \exp(\pg)$ on an inner-product vector space $(V,\ip)$ (see \cite{GIT}) we have that
\begin{equation}\label{eqn_defmm}
  \la \mmm_\G (\mu), A \ra = \frac{1}{\Vert \mu \Vert^2} \, \la \pi(A) \mu, \mu \ra, \qquad \forall A\in \pg.
\end{equation}
Since the data involved in this equation for $\G = \Gl(\vg,J)$ is induced from the corresponding data for $\G = \Gl(\ngo)$, it follows that $\mmm_{\Gl(\vg,J)} (\mu) = \proy(\mmm_{\Gl(\ngo)} (\mu))$, where $\proy : \pg \to \pg(\vg,J)$ denotes orthogonal projection. This fact, together with \eqref{eqn_Pmunil} and \eqref{eqn_Ricnil} yield the desired formula for $P_\mu$. The claim for the trace follows by applying \eqref{eqn_defmm} for $A = \proy_\vg : \ngo \to \ngo$ the orthogonal projection onto $\vg$, and using that $\pi(\proy_\vg)\mu = -2 \mu$ for a two-step nilpotent $\mu$. 
\end{proof}

As explained in Section \ref{sec_prelim} (see Lemma \ref{lem_BFnormeqn}), the normalized solution $\nu(t) :=  \, \mu(t)/\Vert \mu(t) \Vert$ can be interpreted as a solution of the normalized bracket flow equation
\[
  \ddt \nu = - \pi(P_\nu + r_\nu \Id_\ngo) \nu,
\]
where the scalar functional $r_\nu$ is given by $r_\nu = \tr (P_\nu \mmm_{\Gl(\ngo)}(\nu))$. Lemma \ref{lem_nilPmumm} implies that $r_\nu = \unc \,  \Vert \mmm_{\Gl(\vg,J)}(\nu)\Vert^2$, and from this we can easily deduce the following

\begin{corollary}\label{cor_BFgradf}
The norm-normalized pluriclosed bracket flow for $2$-step nilpotent Lie brackets coincides, up to time reparameterization, with the negative gradient flow of the real-analytic functional
\[
    F(\nu) = \big\Vert \mmm_{\Gl(\vg,J)}(\nu) \big\Vert^2 = \frac{16}{\Vert \nu\Vert^4} \cdot  {\Vert P_\nu \Vert^2}.
\]
\end{corollary}

\begin{proof}
According to \cite[Lemma 7.2]{GIT}, if $\mmm(\nu) := \mmm_{\Gl(\vg,J)}(\nu)$ then we have that
\[
  (\nabla F)_\nu = 4 \, \pi \big( \mmm(\nu) + \Vert \mmm(\nu) \Vert^2 \, \Id_\ngo \big) \, \nu.
\]
On the other hand, the above remarks imply that the norm-normalized bracket flow equation is given by
\[
    \ddt \nu = - \frac14 \pi( \mmm(\nu) + \Vert \mmm(\nu) \Vert^2 \, \Id_\ngo) \, \nu.
\]
Since both vector fields differ only by a scalar multiple, the corollary follows.
\end{proof}

We are now in a position to prove the main result of this section, and to deduce Theorem \ref{mainthm_nil} from it.

\begin{proof}[Proof of Theorem \ref{thm_limitnil}]
By compactness of the unit sphere in $V(\ngo)$, the family $\nu(t) := \mu(t) / \Vert \mu(t)\Vert$ has an accumulation point $ \nu_\infty$.
By Corollary \ref{cor_BFgradf} and {\L}ojasiewicz's theorem on real-analytic gradient flows \cite{Loj63}, the length of the curve $[0,\infty) \to V(\ngo)$, $t \mapsto \nu(t)$, is finite. In particular, $\nu(t) \to  \nu_\infty$ as $t\to \infty$, and $\nu_\infty$ is a fixed point of the norm-normalized bracket flow equation~ \eqref{eqn_BFnorm}. Since in this case there is no gauge, i.e.~$U_{\nu(t)} \equiv 0$, this implies that
\[
  P_{\nu_\infty} + r_{\nu_\infty} \Id_\ngo \in \Der(\nu_\infty).
\]
Therefore, $\nu_\infty$ is an algebraic pluriclosed soliton (see Definition \ref{def_soliton}). It is non-flat because  $\tr P_{\nu_\infty} = -\unc \neq 0$ by Lemma \ref{lem_nilPmumm}.
\end{proof}

\begin{proof}[Proof of Theorem \ref{mainthm_nil}]
In the notation of Theorem \ref{thm_limitnil} we have that $\nu(t):=\mu(t)/\Vert \mu(t)\Vert \to \mu_\infty$ as $t\to\infty$, where $\mu_\infty$ is a nilpotent Lie bracket (possibly non-isomorphic to $\mu(0)$) corresponding to an algebraic pluriclosed soliton on a simply-connected nilpotent Lie group $\N_{\mu_\infty}$. By \cite[Cor.~6.20(v)]{spacehm} this yields $C^\infty$ convergence of $(\N_{\nu(t)}, g_{\nu(t)})$ to $(\N_{\mu_\infty}, g_{\mu_\infty})$ (recall that $\N_{\nu(t)}$ and $\N_{\mu\infty}$ are diffeomorphic to $\RR^{2n}$, so convergence hear simply means $C^{\infty}$ uniformly over compact subsets of $\RR^{2n}$). Theorem \ref{mainthm_nil} would now follow from this and Theorem \ref{thm_gaugedBF} (no gauging), provided we show that $\Vert \mu(t) \Vert \sim  t^{-1/2}$ as $t\to \infty$. This is a well-known property of the moment map's negative gradient flow, but for convenience of the reader we provide a proof in our particular case.

We compute the evolution equation for $\Vert \mu(t) \Vert^2$ along the pluriclosed bracket flow \eqref{eqn_BFnil}. And for that we use Lemma \ref{lem_nilPmumm}, which gives the formula $\la P_\mu, A \ra = \unc \, \la \pi(A)\mu, \mu \ra$:
\[
	\ddt \Vert \mu \Vert^2  = 2 \, \big\la \ddt \mu, \mu \big\ra = - 2 \, \la \pi(P_\mu) \mu, \mu\ra  = - 8 \, \Vert P_\mu \Vert^2. 
\]
Now on one hand, $\Vert P_\mu \Vert^2$ can be bounded below by $(\dim \vg)^{-1} \, (\tr P_\mu)^2$, and again by Lemma \ref{lem_nilPmumm} the latter equals $(4 \dim \vg)^{-1} \, \Vert \mu \Vert^4$. On the other hand, 
\[
	\Vert P_\mu \Vert^2 = \unc \, \la \pi(P_\mu) \mu, \mu \ra \leq \unc \, C_\pi \, \Vert P_\mu \Vert \, \mu^2,
\]
where $C_\pi >0$ is some constant depending only on the dimension. The above gives an upper bound for $\Vert P_\mu \Vert^2$, of the form $C \, \Vert \mu \Vert^4$. Therefore, the evolution equation for $\Vert \mu \Vert^2$ can be compared in both directions with an ODE of the form $y' = - c \, y^2$, and from this it is clear that $\Vert \mu(t) \Vert^2 \sim t^{-1}$ as $t\to \infty$, thus concluding the proof.
\end{proof}

\section{Almost-abelian solvmanifolds}\label{sec_almostabelian}

In this section we study the existence of SKT metrics and their evolution under the pluriclosed flow on a simply-connected, \emph{almost-abelian} Lie group $\G$, i.e.~ one whose Lie algebra $\ggo$ has a codimension-one abelian ideal $\ngo$. Notice that such $\G$ is in particular solvable.

If $(J,g)$ is a left-invariant Hermitian structure on $\G$, one easily sees that there exists an orthonormal basis $\{e_1, \ldots , e_{2n}\}$ for $\ggo$ such that
\[
\ngo = \vspan_\RR \la e_1, \ldots , e_{2n-1} \ra , \quad J e_1 = e_{2n}, \quad J (\ngo_1) \subset \ngo_1,
\]
where $\ngo_1 = \vspan_\RR \la e_2, \ldots, e_{2n-1} \ra$. We also set $J_1 :=  J |_{\ngo_1}$.

\begin{lemma}\label{lem_integrmuA}
The complex structure $J$ is integrable if and only if $\ad e_{2n}$ leaves $\ngo_1$ invariant, and $A := (\ad e_{2n})|_{\ngo_1}$ commutes with $J_1.$ 
If this is the case, then $\ad e_{2n}$ is given by
\begin{equation}\label{eqn_ade2n}
  \ad e_{2n} = \threematrix{a}{0}{0}{v}{A}{0}{0}{0}{0},\qquad a\in \RR, \quad v\in \ngo_1, \quad A \in \glg(\ngo_1),  \quad [A,J_1]=0.
\end{equation}
\end{lemma}
\begin{proof}
One easily checks that the vanishing of the Nijenhuis tensor reduces to
\[
  0 = N_J(e_1, e_i) =  [e_1,e_i] + J [J e_1,e_i] + J [e_1, J e_i] -  [J e_1, J e_i]  = J [e_{2n},e_i] -  [e_{2n}, J e_i],
\]
for all $i=2, \ldots, 2n-1$. If this is the case then $[e_{2n}, J e_i] =J [e_{2n},e_i] \in \ngo \cap J \ngo = \ngo_1$, from which $\ad e_{2n}$ preserves $\ngo_1$ and $A$ commutes with $J_1$. The converse is clear.
\end{proof}


\begin{notation}\label{not_skt}
We fix now and for the rest of this section a real inner product vector space $(\ggo, \ip)$ with an orthogonal decomposition
\begin{equation}\label{eqn_decggo}
   \ggo = \RR e_1 \oplus \ngo_1 \oplus \RR e_{2n}, \qquad \ngo := \RR e_1 \oplus \ngo_1,
\end{equation}
with $e_1, e_{2n}$ unitary. We also fix a $\ip$-compatible complex structure $J$ preserving $\ngo_1$ and with $J e_1 = e_{2n}$. We denote by $\mu = \mu(a,v,A)$ an almost-abelian Lie bracket on $\ggo$ with $\ad e_{2n}$ given as in \eqref{eqn_ade2n}, so that $J$ defines an integrable complex structure on the Lie algebra $(\ggo, \mu)$.
\end{notation}

\begin{example}\label{ex_t2inoue}
In \cite[$\S$3]{FT09} the authors study a $6$-dimensional, almost-abelian solvable Lie algebra $\sg_{a,b}$ which according to Notation \ref{not_skt} may be described as $(\RR^6,\mu(a,v,A))$, with
\[
   v = \left(  \begin{array}{c} 0 \\ 0 \\ 0 \\ 0 \end{array} \right), \qquad
    A = \left(\begin{array}{cccc}  -\tfrac{a}{2} & 0 & 0 & 0 \\ 0 & 0 & b & 0 \\  0 & -b & 0 & 0 \\ 0 & 0 & 0 &  -\tfrac{a}{2} \end{array} \right), \qquad a,b  \in \RR \backslash \{ 0\}.
\]
Notice that we have rearranged the basis for consistency within this article. The complex structure is given by $J e_1 = e_6$, $Je_2 = e_5,$ $Je_3 = e_4$, thus by Lemma \ref{lem_integrmuA}, it defines an integrable complex structure on the corresponding simply-connected Lie group $\Ss_{a,b}$. The group $\Ss_{1,\frac{\pi}{2}}$ admits a cocompact lattice $\Gamma$, and the compact manifold $M^6 = \Gamma \backslash \Ss_{1,\frac{\pi}{2}}$ is the total space of a $\mathbb{T}^2$-bundle over an Inoue surface $S_M$.
\end{example}

\subsection{SKT almost-abelian solvmanifolds}\label{sec_SKTaa}

In this section we determine necessary and sufficient conditions for an almost-abelian Lie group $\G$ to admit a left-invariant SKT metric. Following Notation \ref{not_skt}, we will write $\mu$ as short for the infinitesimal data $(\ggo, \ip, J, \mu)$. By Proposition~ \ref{prop_movingbrackets}, the orbit $\Gl(\ggo, J) \cdot \mu \subset V(\ggo)$ parameterizes all left-invariant Hermitian structures on $\G$. In this way, the SKT condition reduces to an algebraic condition on the Lie bracket:

\begin{lemma}\label{lem_SKT}
If $\mu = \mu(a,v,A)$ is almost-abelian then  the metric $g$ is SKT if and only if
\begin{equation}\label{eqn_SKTcondition}
  aA+A^2+A^tA \in \sog(\ngo_1),
\end{equation}
where $(\cdot)^t$ denotes the transpose with respect to $\ip$.
\end{lemma}
\begin{proof}
We apply the formulae for $c$ and $dc$ given in \eqref{eqn_formulac}, \eqref{eqn_formuladc}. Since $\ngo_1$ is an abelian ideal preserved by $J$, it is clear that $c(X,Y,Z) = 0$ when all the entries lie in $\ngo_1$. One quickly checks that this implies that $dc(W,X,Y,Z) = 0$ if three of the entries lie in $\ngo_1$. Therefore, the SKT condition $dc = 0$ is equivalent to
\[
  dc(e_{2n}, e_1, Y, Z) = 0, \qquad \forall \, Y,Z\in \ngo_1.
\]
Using that $\ngo$ is abelian it is clear that the last three terms in \eqref{eqn_formuladc} vanish, thus
\begin{align*}
  dc(e_{2n}, e_1, Y, Z) &= -c(\mu(e_{2n}, e_1), Y, Z) +  c(\mu(e_{2n},Y), e_1, Z) - c( \mu(e_{2n}, Z), e_1, Y) \\
  & = - a \, c (e_1, Y, Z) - c(v, Y , Z) + c(A Y, e_1, Z) - c(A Z, e_1, Y).
\end{align*}
As above we see that the second term in the right-hand-side vanishes. Regarding the other three terms, notice that for $W, X \in \ngo_1$ we have that
\begin{align*}
  c(e_1, W, X) &=   - \la \mu(e_{2n}, J W), X \ra + \la \mu(e_{2n}, J X), W\ra = - \la A J W, X\ra + \la A J X, W\ra   \\
    &   =  - \la A J W, X\ra - \la J A^t W , X \ra  =  -2 \,  \la S(A) J W, X\ra,
\end{align*}
where $S(A) = \unm(A + A^t)$ denotes the symmetric part of an endomorphism. Recall that $J|_{\ngo_1}$ commutes with $A$ and hence also with $A^t$ since $J = -J^t$. After using the alternancy of $dc$ all three remaining terms are of the type $c(e_1, W, X)$, and we immediately obtain
\[
    dc(e_{2n}, e_1, Y, Z) =  \left\la   \big( a S(A) + S(A) A + A^t S(A) \big) J\, Y, Z \right\ra.
\]
The lemma now follows, since $a S(A) + S(A) A + A^t S(A) = S(a A + A^2 + A^t A)$.
\end{proof}

This suggests the following

\begin{definition}
An almost-abelian Lie bracket $\mu = \mu(a,v,A)$ satisfying $aA+A^2+A^tA \in \sog(\ngo_1)$ is called an \emph{SKT bracket}.
\end{definition}

\begin{corollary}\label{cor_trace}
Any SKT bracket $\mu = \mu(a,v,A)$ satisfies $a \tr(A) \leq 0$, with equality if and only if $A\in \sog(\ngo_1)$.
\end{corollary}

\begin{proof}
Taking traces in \eqref{eqn_SKTcondition} yields $a \tr(A) + 2 \tr(S(A)^2)=0$, and $\tr (S(A)^2) \geq 0$ with equality if and only if $S(A) = 0$.
\end{proof}

\begin{remark}\label{rem_Aneq0}
If $\mu$ is almost-abelian and nilpotent, then $\ad e_{2n}$ is nilpotent, thus $a=0$ and $A$ is nilpotent. By the equality condition in Corollary \ref{cor_trace} we must have $A=0$, and then $\mu$ is in fact two-step nilpotent.
\end{remark}

\begin{lemma}\label{lem_realpart}
If $\mu = \mu(a, v, A)$ is an SKT bracket, the real part of the eigenvalues of $A$ is either $0$ or $-\tfrac{a}{2}$.
\end{lemma}

\begin{proof}
Let $\ngo_1 ^\CC = \ngo_1  \otimes_\RR \CC = \ngo_1  \oplus \sqrt{-1}  \, \ngo_1 $ be the complexified vector space, and for $B\in \glg(\ngo_1 )$ denote also by $B\in \glg(\ngo_1 ^\CC)$ the corresponding $\CC$-linear endomorphism. The fixed inner product $\ip$ on $\ngo_1 $ induces a Hermitian inner product $\ipH$ on $\ngo_1 ^\CC$, and $B\in \sog(\ngo_1)$ if and only if
\[
    \laH B z, \bar z \raH = 0, \qquad \forall z\in \ngo_1^\CC.
\]
Let $\lambda \in \CC$ be an eigenvalue of $A$ with corresponding eigenvector $z = u + \sqrt{-1} v \in \ngo_1^\CC$, $u,v \in \ngo_1$. Applying the above for $B = a A + A^2 + A^t A$, and Lemma \ref{lem_SKT}, we have that
\[
     0 = \laH \big( a A + A^2 + A^t A \big) z, \bar z  \raH = \lambda (a + 2 \lambda) \cdot \laH z, \bar z\raH.
\]
If $\lambda \neq 0, -\tfrac{a}{2}$ then $\laH z, \bar z\raH = 0$. This is equivalent to saying that $\Vert u \Vert = \Vert v \Vert$ (say equal to $1$) and $\la u,v\ra = 0$. In this case, setting $\lambda = \alpha + \sqrt{-1} \beta$, and using that $A u = \alpha u - \beta v$, $Av = \beta u + \alpha v$, and the SKT condition, one gets
\[
   0 = \big\la \big( a A + A^2 + A^t A\big) u, u \big\ra = a \, \alpha + (\alpha^2 - \beta^2) + (\alpha^2 + \beta^2) = \alpha  (a + 2 \alpha),
\]
from which the lemma follows.
\end{proof}

We are now in a position to completely characterize the SKT condition:

\begin{theorem}\label{thm_SKT}
An almost-abelian Lie bracket $\mu = \mu(a,v,A)$ is SKT if and only if $[A,A^t] = [A,J_1] = 0$ and each eigenvalue of $A$ has real part equal to $0$ or $-\tfrac{a}{2}$.
\end{theorem}

\begin{proof}
Integrability of $J$ is equivalent to $J(\ngo_1) \subset \ngo_1$ and $[A, J_1] = 0$ thanks to Lemma \ref{lem_integrmuA}.

Assume first that $\mu$ is SKT. Since $A$ commutes with $J_1$, by Lemma \ref{lem_realpart} its eigenvalues $\lambda_1, \ldots, \lambda_{2n-2}$ come in pairs and can be rearranged so that
\[
  \Re(\lambda_1) = \cdots = \Re(\lambda_{2k}) = -\tfrac{a}{2}, \qquad \Re(\lambda_{2k+1}) = \cdots = \Re(\lambda_{2n-2}) = 0.
\]
On the other hand, taking traces in \eqref{eqn_SKTcondition} one obtains $\Vert S(A)\Vert^2 = \unm k \, a^2$.
This yields equality in Corollary \ref{cor_app},
thus $A$ is a normal endomorphism.

The converse assertion follows by direct computation using Lemma \ref{lem_SKT}.
\end{proof}

We conclude this section with a remark about generalized K\"ahler structures. Recall that a generalized K\"ahler manifold is a Riemannian manifold $(M^{2n},g)$ together with two $g$-compatible complex structures $J_+, J_-$ satisfying
\begin{equation}\label{eqn_genK}
    d^c_+ \omega_+ =  - d^c_- \omega_- =: H; \qquad d H = 0,
\end{equation}
see e.g.~\cite{ST12}, \cite{Gua14}. Here, $d^c_\pm = \sqrt{-1} \big(\overline\partial_\pm - \partial_\pm \big) = (-1)^r J_\pm d_\pm J_\pm$ on $r$-forms, with $(J \alpha)(\cdot, \ldots, \cdot) = (-1)^r \alpha(J\cdot, \ldots, J \cdot)$ for an $r$-form $\alpha$.  The $3$-form $H$ is called the torsion. Since $d d^c  = 2 \sqrt{-1} \partial \overline\partial$, a generalized K\"ahler manifold $(M,g,J_\pm)$ may be thought of as a pair of pluriclosed structures whose corresponding Riemannian metrics coincide, and which are compatible with each other in the sense that they satisfy the first equation in \eqref{eqn_genK}.

A natural question that arises in our context is  to determine which almost-abelian SKT brackets are compatible with a generalized K\"ahler structure. To that end, consider on $\ggo = \RR e_1 \oplus \ngo_1 \oplus \RR e_{2n}$ the following pair complex structures:
\[
  J_+=  \threematrix{0}{0}{-1}{0}{J_1}{0}{1}{0}{0}, \quad J_- = \threematrix{0}{0}{-1}{0}{-J_1}{0}{1}{0}{0},
\]
for some fixed $J_1$ on $\ngo_1$ with $J_1^2 = - \Id_{\ngo_1}$, compatible with $\ip$. For each almost-abelian Lie bracket $\mu = \mu(a,v,A)$ on $\ggo$ denote with the same names the left-invariant almost-complex structures defined by $J_+$, $J_-$ on the simply-connected Lie group $\G_\mu$ with Lie algebra $(\ggo,\mu)$. Notice that by Lemma \ref{lem_integrmuA}, $J_+$ is integrable if and only if $J_-$ is so, and in what follows we assume both of them to be integrable, i.e. $[A,J_1] = 0$.

\begin{proposition}\label{prop_genk}
An almost-abelian Lie bracket  $\mu=\mu(a,v,A)$ is compatible with a generalized K\"ahler structure if and only if $a A + A^2+ A^t A \in \sog(\ngo_1)$ and $v=0.$
\end{proposition}
\begin{proof}
By Lemma \ref{lem_SKT} the SKT condition for both complex structures is equivalent to $a A + A^2+ A^t A \in \sog(\ngo_1).$ Recall that $c_\pm = d_\pm^c \omega_\pm$, thus we need to check whether $c_+ + c_- = 0$. In order to compute $c_+ + c_-,$ we apply formula (\ref{eqn_formulac}). Notice that if either two or three of the entries lie in $\ngo_1,$ then $c_+ + c_-$ vanishes. Thus,  $c_+ + c_- = 0$ is equivalent to $(c_+ + c_-)(e_1, e_{2n},Z)=0,$ for all $Z\ \in \ngo_1$. Using the notation from Lemma \ref{lem_integrmuA} this becomes $\la v, Z \ra =0$ for all $Z\in \ngo_1,$ and the result follows.
\end{proof}

It follows that the solvable Lie group $\Ss_{a,b}$ from Example \ref{ex_t2inoue} admits a left-invariant generalized K\"ahler structure, a fact  first observed in \cite{FT09}.

\subsection{The Bismut-Ricci form of almost-abelian solvmanifolds}

In this section we obtain a formula for the Bismut-Ricci form $\rho^B$ of a left-invariant Hermitian structure on an almost-abelian Lie group $\G$ which is not necessarily unimodular.

Let $(\ggo, \mu)$ be a $2n$-dimensional real Lie algebra endowed with an integrable Hermitian structure $(J,\ip)$. Let $\{ e_i\}_{i=1}^{2n}$ be a $\ip$-orthonormal basis such that $J e_i = e_{2n+1-i}$ for $i=1,\ldots, n$. The fundamental form $\omega = g(J\cdot ,\cdot)$ is given by
\[
  \omega = e^1 \wedge e^{2n} + \cdots + e^n \wedge e^{n+1},
\]
where $\{ e^i\}$ is the dual basis. We identify  $\ggo^*$ with the left-invariant 1-forms on $\G$. For any $\alpha \in \ggo^*$, $d \alpha$ is a left-invariant 2-form, determined by its values on $\Lambda^2 \ggo$. These are given by $ d \alpha(X,Y) = - \alpha (\mu(X,Y))$, for $X, Y\in \ggo$.

Let $(\cdot)^{\flat} : \ggo \to \ggo^*$ denote the usual isomorphism induced by $\ip$: $X^\flat(\cdot) := \la X, \cdot\ra$, $X\in \ggo$.

\begin{proposition}\label{prop_BisRicaa}
The Bismut-Ricci form $\rho^B_\mu$ of an almost-abelian Lie bracket $\mu = \mu(a,v,A)$ is given by
\[
  \rho^B_\mu =  - (a^2 + \unm a \tr A + \Vert v\Vert^2 ) \, e^1 \wedge e^{2n} - (A^t v)^\flat \wedge e^{2n}.
\]
\end{proposition}

\begin{proof}
By \cite{Vez13}, $\rho^B_{\mu}$ can be locally written as the derivative of the $1$-form $\theta_\mu^{-1}\in \ggo^*$ given by

\[
  \theta_\mu^{-1}(X) = -\unm \big(\tr(J \ad_\mu X ) + \tr (\ad_\mu J X) + 2 \la\omega, dX^{\flat}\ra  \big).
\]
The first two summands above vanish on $\ngo_1 = \vspan\{e_2,\ldots,e_{2n-1}\}$, whereas
\[
    \tr (J \ad_\mu e_1) + \tr \ad_\mu J e_1 = -a + \tr \ad_\mu e_{2n} = \tr A.
\]
Also, a straightforward computation yields
\[
     \la\omega, dX^{\flat}\ra   = \sum_{j=1}^n \mu(e_{2n+1-j}, e_j)^{\flat}(X) = \mu(e_{2n}, e_1)^\flat  (X).
\]
Using that $\rho_\mu^B (X,Y) = d \theta_\mu^{-1}(X,Y) = - \theta^{-1}_\mu\big(\mu(X,Y) \big)$ and the almost-abelian condition we see that $\rho_\mu^B$ vanishes on $\ngo \wedge \ngo$. For $Y = e_{2n}$, $X = e_j$ with $j=2,\ldots,2n-1$ we have that $\mu(X,Y)\in \ngo$, hence
\[
    \rho_\mu^B(e_j, e_{2n}) = \la \mu(e_{2n}, e_1), \mu(e_j, e_{2n})\ra = -\la ae_1 + v, A e_j\ra = -\la A^t v, e_j\ra.
\]
On the other hand, for $Y=e_{2n}$, $X = e_1$ we compute directly to obtain
\[
    \rho_\mu^B(e_1, e_{2n}) = -\theta_\mu^{-1}(-ae_1 - v) = - \unm a  \tr A - \la \mu(e_{2n},e_1), ae_1 + v\ra = -(a^2 + \unm a \tr A + \Vert v \Vert^2).
\]
\end{proof}

\subsection{The pluriclosed flow on almost-abelian solvmanifolds}

In this section we study the pluriclosed flow of SKT metrics on almost-abelian Lie groups by using the bracket flow approach introduced in Section \ref{section_BF}.

Let $(\G, J, g_0)$ be an almost-abelian, simply-connected Lie group with a left-invariant SKT structure, denote by $J$, $\ip$ the corresponding tensors on $\ggo$, and let $\mu_0$ be the Lie bracket of $\ggo$. Let $g(t)$ be the solution to the pluriclosed flow equation \eqref{eqn_pluriflow} with initial condition $g_0$.  According to \cite[Thm.~5.1]{Lau15}, $g(t)$ coincides up to pull-back by biholomorphisms with the family of Hermitian manifolds determined by the brackets $\mu(t)$ solving
\begin{equation}\label{eqn_BFaa}
    \ddt \mu = - \pi(P_\mu) \mu, \qquad \mu(0) = \mu_0,
\end{equation}
where $P_\mu \in \End(\ggo)$ is the endomorphism associated to $\unm (\rho_\mu^B)^{1,1}$ via $\omega$, see \eqref{eqn_defPmu}.

\begin{lemma}\label{lem_Pmuaa}
For an almost-abelian SKT bracket $\mu = \mu(a,v,A)$, $P_\mu$ is given by
\[
  P_{\mu} =  \threematrix{c}{w^t}{0}{w}{0}{J_1 w}{0}{-w^t J_1}{c}, \qquad c = \big(\tfrac{k}{4} - \unm\big) a^2 - \unm \Vert v\Vert^2, \qquad w = -\unc A^t v \in \ngo_1.
\]
where the blocks are according to \eqref{eqn_decggo}. Here $2  k$ is the multiplicity of $-a/2$ as an eigenvalue of $S(A) = \unm(A + A^t)$, see Theorem \ref{thm_SKT}.
\end{lemma}

\begin{proof}
Notice that for any $\alpha, \beta \in \ggo^*$ we have
\[
  ( \alpha \wedge \beta) (J \cdot, J \cdot) = (J \alpha) \wedge (J \beta) (\cdot ,\cdot),
\]
where $J$ acts on $\ggo^*$ by $J \alpha = \alpha\circ J^{-1} =  - \alpha \circ J$. Thus,
\[
    (\alpha\wedge \beta)^{1,1} = \unm  (\alpha\wedge \beta)  + \unm (J \alpha) \wedge (J\beta).
\]
Since $\tr A = - k a$, from Proposition \ref{prop_BisRicaa} we have that
\begin{align*}
    \unm \, \big(\rho_\mu^B \big)^{1,1} &=  c \cdot \left( e^1 \wedge e^{2n}  \right)^{1,1} + \left( (2 w)^\flat \wedge e^{2n} \right)^{1,1} \\
        &=  \tfrac{c}{2} \cdot e^1 \wedge e^{2n} - \tfrac{c}{2} \cdot e^{2n} \wedge e^1 + w^\flat \wedge e^{2n} - (J w)^\flat \wedge e^1 \\
        &= c \cdot e^1 \wedge e^{2n}  + w^\flat \wedge e^{2n} + e^1 \wedge (Jw)^\flat.
\end{align*}
with $c, w$ as in the statement, and we have used that $J e^1 = e^{2n}$ and $J(X^\flat) = (J X)^\flat$ for $X\in \ggo$. By a routine computation one checks that the last expression coincides with $\omega(P_\mu \cdot, \cdot)$.
\end{proof}

\begin{remark}\label{rem_staticaa}
Notice that from Lemma \ref{lem_Pmuaa} it immediately follows that for dimension $\geq4$ there are no static almost-abelian SKT brackets with $\alpha \neq 0$. Regarding pluriclosed algebraic solitons, since a derivation $D$ of a non-nilpotent almost-abelian Lie algebra $\mu(a,v,A)$ maps $\ggo$ into the nilradical $\ngo$, it follows from \eqref{eqn_defsoliton} and Lemma \ref{lem_Pmuaa} that the cosmological constant $\alpha$ of the soliton equals $c = (\tfrac{k}{4} - \unm)a^2 - \unm \Vert v\Vert^2$.
\end{remark}

Let $\mu(t)$ be the solution to \eqref{eqn_BFaa}. Since $P_\mu$ does not preserve the nilradical $\ngo$ of $\mu_0$, the nilradical of $\mu(t)$ will not be $\ngo$ for $t > 0$. Thus the set of brackets of the form $\mu = \mu(a,v,A)$ (which are defined in terms of the fixed decomposition \eqref{eqn_decggo}) is \emph{not} invariant under the bracket flow ODE~ \eqref{eqn_BFaa}. To overcome this issue we need to find the right \emph{gauge}. For each $\mu = \mu(a,v,A)$  consider
\[
U_\mu :=  \threematrix{0}{w^t}{0}{-w}{\tfrac{a}{4}(A-A^t)}{-J_1 w}{0}{- w^t J_1}{0}, \qquad w = -\unc A^t v.
\]

\begin{lemma}
We have that $U_\mu \in \ug(\ggo, J)$.
\end{lemma}
\begin{proof}
Notice that $U_\mu^t = - U_\mu$, and since $[A,J_1] = 0$ by Lemma \ref{lem_integrmuA}, we also have $[U_\mu, J] = 0$.
\end{proof}

By Theorem \ref{thm_gaugedBF}, we may study instead the solutions to the gauged bracket flow equation
\begin{equation}\label{eqn_gBFaa}
    \ddt \bar\mu = -\pi(P_{\bar\mu} - U_{\bar\mu}) {\bar\mu}, \qquad \bar\mu(0) = \mu_0.
\end{equation}

\begin{proposition}\label{prop_odeaa}
For an SKT almost-abelian initial condition $\mu_0 = \mu(a_0, v_0, A_0)$, the gauged bracket flow equation \eqref{eqn_gBFaa} is equivalent to the ODE system
\begin{equation}\label{eqn_gbfaa}
  \begin{cases}
        a' = c \, a, \\
        v' = c \, v + S \, v - \unm \Vert v\Vert^2 v, \\
        A' = c \, A,
  \end{cases}
\end{equation}
where $c = (\tfrac{k}{4} - \unm) a^2 - \unm \Vert v\Vert^2 \in \RR$, $2 \, k = \rank (A + A^t)$, and
\[
  S = S(a,A) = \big(\tfrac{k}{4} - \unm \big) a^2 \, \Id_{\ngo_1} - \unm A A^t + \tfrac{a}{4} (A+A^t).
\]
Moreover, the solution $\bar\mu(t) = \mu(a(t), v(t), A(t))$ to \eqref{eqn_gBFaa} consists entirely of SKT brackets.
\end{proposition}

\begin{proof}
Notice that for any almost-abelian $\mu = \mu(a,v,A)$ we have
\[
  P_\mu - U_\mu = \threematrix{c}{0}{0}{2w}{\tfrac{a}{4}(A^t-A)}{2 J_1 w}{0}{0}{c},
\]
which preserves the subspaces $\ngo$ and $\ngo_1$. It follows that for the curve $\{h(t)\}\subset \Gl(\ggo)$ giving the solution $\bar\mu(t) = h(t) \cdot \mu_0$ (see Remark \ref{rmk_ht}) we have that $h(t) (\ngo) \subset \ngo$ and $h(t) (\ngo_1) \subset \ngo_1$. Thus, $\ngo$ is also a codimension-one ideal of $\bar\mu(t)$ for all $t$, and $\ngo_1$ is still preserved by the adjoint action of $e_{2n}$ (and invariant by $J$, since $J$ is fixed). Hence $\bar\mu(t)$ is of the form $\mu(a(t), v(t), A(t))$. To determine the evolution equations for $a$, $v$ and $A$ one uses \eqref{eqn_ade2n} and notices that
\[
    (\ad_{\bar\mu} e_{2n})' = \ad_{\bar\mu'} e_{2n} = - \ad_{\pi(P_{\bar\mu} - U_{\bar\mu}) \bar\mu} e_{2n} = -[(P_{\bar\mu}-U_{\bar\mu}), \ad_{\bar\mu} e_{2n}] + {\ad_{\bar\mu} } \big( (P_{\bar\mu} - U_{\bar\mu}) e_{2n} \big).
\]
Restricting to $\ngo$ yields
\[
  \twomatrix{a}{0}{v}{A}'  =
    - \left[  \twomatrix{c}{0}{2 w}{\tfrac{a}{4} (A^t-A)} , \twomatrix{a}{0}{v}{A} \right] + c \twomatrix{a}{0}{v}{A},
\]
from which one gets the evolution equations
\[
  \begin{cases}
        a' = c \, a, \\
        v' = 2 \, c \, v - 2 a w  - \tfrac{a}{4} (A^t-A) v + 2 A w , \\
        A' = c \, A   + \tfrac{a}{4} [(A - A^t), A].
  \end{cases}
\]
Independently of what $a(t)$ and $v(t)$ are, it follows from uniqueness of ODE solutions that for an initial condition with $A_0$ normal, $A(t)$ will stay normal for all $t$, and it will evolve by $A' = c A$. By Theorem \ref{thm_SKT} this implies that, as expected, the SKT condition is preserved along the flow. The evolution equation for $v$ follows by using that $w = -\unc A^t v$.
\end{proof}

The following results will be important for the analysis of the ODE \eqref{eqn_gbfaa}:

\begin{lemma}\label{lem_S}
The symmetric map $S = S(a,A)$ from Proposition \ref{prop_odeaa} satisfies
\[
    \la S u, u\ra \leq  \big(\tfrac{k}{4} - \unm \big) \, a^2 \, \Vert u\Vert^2,
\]
for all $u\in \ngo_1$, with equality if and only if $u \in \ker A$.
\end{lemma}

\begin{proof}
Recall that $S = (\tfrac{k}{4} - \unm) a^2 \Id_{\ngo_1} - \unm A A^t + \tfrac{a}{4} (A+A^t)$. Let $u\in \ngo_1$ be an eigenvector of $\unm (A+A^t)$, with eigenvalue $\lambda \in \{ -a/2 , 0\}$ by Theorem \ref{thm_SKT}. Using that $A$ is normal we get
\[
      \Big\la { \big(-\unm AA^t + \tfrac{a}{4} (A+A^t) \big) } u, u \Big\ra  = -\unm \, \Vert A u \Vert^2 + \tfrac{a}{2} \, \lambda \, \Vert u \Vert^2 \leq 0,
\]
and equality holds if and only if $Au = 0$.
\end{proof}

\begin{lemma}\label{lem_beta}
Fix $\delta \geq 1$, $\beta > 0$  and consider a $C^1$ function $y : (t_0,\infty) \to (0,\infty)$ satisfying
\[
    2 \, y \, (\beta - \delta y)  \, \leq  \,  y'  \, \leq  \,  2 \, y \, (\beta - \delta^{-1} y),
\]
for all sufficiently large $t$. Then, the $\omega$-limit of $y(t)$ is contained in $[\delta^{-1} \beta, \delta \beta]$.
\end{lemma}

\begin{proof}
Assume that $l := \liminf_{t\to\infty} y(t) < \delta^{-1} \beta$. Notice that for all $t$ such that $y(t) < \delta^{-1} \beta$ one has $y'(t) > 0$ by the differential inequality. Thus, in order to have $l < \delta^{-1}\beta$ the only possibility is that $y(t) < l$ for all $t$ and $y(t) \nearrow l$ as $t\to\infty$. But then there is a positive lower bound $\beta - \delta y(t) \geq \epsilon > 0$ for all $t>t_0$, yielding $y' \geq 2 \epsilon y$. Integrating we see that $y(t)$ blows up, contradicting the fact that $y(t) < l$ for all $t$.

The upper bound for $\limsup_{t\to \infty} y(t)$ follows analogously.
\end{proof}

We are now in a position to prove the main result of this section:

\begin{theorem}\label{thm_aa}
The pluriclosed flow of invariant SKT structures on a non-nilpotent, almost-abelian Lie group $\G$ is equivalent to the ODE given in \eqref{eqn_gbfaa}. For a maximal solution $\big(\bar\mu(t) =   \mu(a(t), v(t), \allowbreak A(t) ) \big)_{t\in [0,T)}$ with initial condition $\mu_0 = \mu(a_0,v_0, A_0)$ set $k := k(\G) := \unm \rank (A_0 + A_0^t) \in \ZZ_{\geq 0}$. Then, the behavior of $\bar\mu(t)$ and the corresponding normalized solution $\bar\mu(t)/ \Vert \bar\mu(t)\Vert$ are described in Table \ref{tab_almostabelian}. In every case, the normalized limits are pluriclosed algebraic solitons.
\end{theorem}




{\small
\begin{table}
\[
\begin{array}{ccccccc}
  \mbox{Case}     & k   & \mbox{Constraints}  & \mbox{Unimodular} & T       &  \lim_{t\to T}\bar\mu(t)  & \lim_{t\to T} \bar\mu(t)/\Vert \bar\mu(t)\Vert    \\
    \hline   &&&&&&\\
  \mbox{(i)}      & 0   & a_0=0                   & \checkmark        & +\infty & \mu_\infty                     & \mbox{K\"ahler, Ricci-flat}  \\ \\
    \mbox{(ii)}      & 0   & a_0\neq 0                   & -        & +\infty & 0                     & \mbox{expanding soliton} \\ \\
  \mbox{(iii)}     & 1   & -                   & \checkmark        & +\infty & 0                     & \mbox{expanding soliton} \\ \\
  \mbox{(iv)}    & 2   & -                   &  -                & +\infty & \mu_\infty            & \mbox{steady soliton} \\ \\
  \mbox{(v)}     & >2  & v_0 \notin \Im A_0  &  -                & +\infty & \mu_\infty \neq 0     & \mbox{steady soliton}   \\ \\
  \mbox{(vi)}      & >2  & v_0 \in \Im A_0     &  -                & <\infty & \infty                & \mbox{shrinking soliton}  \\
  &&&&&&\\
  \hline
\end{array}
\]
\caption{The pluriclosed flow on almost-abelian Lie groups}\label{tab_almostabelian}
\end{table}}



\begin{proof}
The first claim follows from Theorem \ref{thm_gaugedBF} and Proposition \ref{prop_odeaa}.

Turning to the ODE analysis, observe first that an initial condition with $(a_0,A_0) = (0,0)$ corresponds to a two-step nilpotent $\G$, and these were analyzed in Section \ref{sec_nilmfd}. Because of this, from now on we assume $(a_0,A_0) \neq (0,0)$.

\vskip5pt

{\bf \noindent Normalized flow.} We consider first the normalized flow $\nu_a(t)$ keeping $a^2 + \Vert A \Vert^2$ constant. As is well-known, up to a time reparameterization, the latter solves an ODE defined by a vector field which coincides with the one in \eqref{eqn_gbfaa} up to adding a multiple of the identity. Since substracting $c\cdot \Id$ in \eqref{eqn_gbfaa} yields
\begin{equation}\label{eqn_normgbfaa}
      \begin{cases}
        a' = 0, \\
         \tilde v' =   S  \tilde v - \unm \Vert  \tilde v \Vert^2  \tilde v, \\
        A' = 0,
    \end{cases}
\end{equation}
where $a$ and $A$ remain constant, it follows that $\nu_a(t) = \mu(a_0, \tilde v(t), A_0)$, for $\tilde v(t)$ solving \eqref{eqn_normgbfaa}. Observe that $S$ also remains  constant. Denote its eigenvalues by $\lambda_1 < \ldots < \lambda_m$, and recall that $\lambda_m \leq (\tfrac{k}{4} - \unm) a^2 =: \Lambda$ by Lemma \ref{lem_S}.

Assume that $k\leq 2$, which ammounts to saying that $\Lambda \leq 0$. By \eqref{eqn_normgbfaa} we have
\[
    \ddt \Vert \tilde v \Vert^2  = 2 \, \la S \tilde v,  \tilde v\ra - \Vert \tilde v \Vert^4 \leq 2 \Lambda \Vert \tilde v \Vert^2 - \Vert \tilde v\Vert^4 \leq - \Vert \tilde v\Vert^4,
\]
thus $\tilde v(t) \to 0$ as $t\to\infty$, by comparison with $y' = -y^2$. The limit bracket $\mu(a_0, 0, A_0) \neq 0$ is a fixed point of \eqref{eqn_normgbfaa}, and hence a pluriclosed algebraic soliton by Proposition \ref{prop_solitonfixedpoint}.

Suppose now that $k>2$ and decompose $\tilde v = v_1 + \cdots + v_s$ as a sum of eigenvectors of $S$, $s\leq m$, with $v_i$ eigenvector with eigenvalue $\lambda_i$ and $v_s \neq 0$. Each $v_i$ evolves only by scaling, thus if we set $r_i = \unm \Vert v_i \Vert^2$, the evolution equation \eqref{eqn_normgbfaa} turns out to be equivalent to the coupled system
\begin{equation}\label{eqn_ri}
    r_i ' =  2 \, \lambda_i \, r_i - \Vert \tilde v \Vert^2 \, r_i, \qquad i = 1, \ldots, s,
\end{equation}
where $\Vert \tilde v \Vert^2 = 2 (r_1+ \cdots + r_s)$. Using that $\Vert \tilde v \Vert^2 \geq 2 r_s,$ from \eqref{eqn_ri} we obtain
\[
  r_s'  \leq 2 r_s (\lambda_s - r_s).
\]
Since $r_s(t) \geq \lambda_s$ yields $r'_s(t) < 0$, it follows that $r_s$ is uniformly bounded above. On the other hand, for any $i < s$ \eqref{eqn_ri} gives
\[
    \ddt \log(r_i / r_s)  = 2 \, (\lambda_i - \lambda_s) < 0,
\]
from which $r_i(t) \to 0$ as $t \to \infty$, since $r_s(t)$ is bounded. Hence, $\Vert \tilde v \Vert^2 / 2 r_s \to 1$ as $t \to \infty$. If $\lambda_s \leq 0$ then it easily follows that $r_s(t) \to 0$ as $t\to \infty$, by comparison with $y' = -2y^2$. Thus, $\tilde v(t) \to 0$ as $t\to\infty$ in this case. For $\lambda_s > 0$  we may rewrite \eqref{eqn_ri} as
\[
  r_s '  = 2\,  r_s \Big(\lambda_s - \tfrac{\Vert \tilde v \Vert^2}{2 r_s} \cdot r_s \Big)
\]
and apply Lemma \ref{lem_beta} to conclude that $r_s(t) \to \lambda_s$ as $t\to\infty$. Indeed,  $\Vert \tilde v \Vert^2 / 2 r_s \to 1$, thus for any $\delta > 1$ there exists $t_0 > 0$ such that for all $t>t_0$ the assumptions of Lemma \ref{lem_beta} are satisfied for $y = r_s$, and this implies that the $\omega$-limit of $r_s(t)$ is contained in $[\delta^{-1} \lambda_s, \delta \lambda_s]$, for all $\delta > 1$.
In any case, we have that $\tilde v(t)$ converges as $t\to \infty$ to a fixed point $\tilde v_\infty$ of \eqref{eqn_normgbfaa}, and again the limit bracket $\mu(a_0,  \tilde v_\infty, A_0) \neq 0$ is a pluriclosed algebraic soliton by Proposition \ref{prop_solitonfixedpoint}.

The type of soliton we get in the limit can be detected by the sign of the cosmological constant~ $\alpha$, which by Remark \ref{rem_staticaa} coincides with $\tilde c_\infty := (\tfrac{k}{4} - \unm) a_0^2 -\unm \Vert \tilde v_\infty \Vert^2$, the limit of the $c$ component in $P_{\nu_a(t)}$. The above analysis yields that for $k\leq 2$  the limit satisfies $\tilde v_\infty=0$, and the results stated in Table \ref{tab_almostabelian} follow immediately. Recall that  $k>0$ implies $a_0 \neq 0$. When $k>2$, if $\tilde v_\infty=0$ then $\tilde c_\infty = (\tfrac{k}{4} - \unm) a_0^2 > 0$ and the soliton is shrinking. On the other hand, if $\tilde v_\infty \neq 0$ (which corresponds to the case $\lambda_s > 0$), the arguments in the previous paragraph imply that $\tilde c_\infty = \Lambda - \lambda_s \geq 0$, with equality if and only if we have both $s = m$ and  equality in Lemma \ref{lem_S}. In order for the latter to happen, the starting value $v_0$ must have a non-trivial component in $\ker A_0$, or in other words, $v_0 \notin \Im A_0$.

To conclude the proof of this case we must justify why is it enough to consider the normalized flow $\nu_a(t)$. To see that, notice that if $\nu(t) := \bar \mu(t) / \Vert \bar \mu(t)\Vert$ denotes the norm-normalized solution, then the fact that $\nu_a(t) \to \nu_\infty \neq 0$  implies that
\[
  \nu(t) = \frac{\nu_a(t)}{\Vert \nu_a(t)\Vert} \underset{t\to\infty}\longrightarrow \frac{\nu_\infty}{\Vert \nu_\infty \Vert},
\]
which differs from $\nu_\infty$ only by scaling.

\vskip5pt
{\bf \noindent Unnormalized flow. } Let us first  assume $k\leq 2$, so that $c\leq 0$ holds. From \eqref{eqn_gbfaa} we have
\[
	\ddt (a^2 + \Vert A \Vert^2 ) = 2 \, c \, (a^2 + \Vert A \Vert^2) \leq 0,
\]
and it follows that $a^2 + \Vert A \Vert^2$ is bounded. On the other hand, from the normalized flow analysis we know that $\Vert v \Vert^2 / (a^2 + \Vert A \Vert^2) \to 0$ as $t \to T$, which yields $v\to 0$. Hence, the bracket is bounded along the solution and $T=+\infty$. Since $a$ and $A$ evolve only by homotheties, and $c\leq 0$,  it is clear that there will be a limit bracket $\mu_\infty$.  For the cases (ii) and (iii) along the normalized flow $c$ converges to a negative constant. Translating this into the unnormalized flow yields
\[
	\frac{c}{a^2 + \Vert A \Vert^2} \leq - \epsilon < 0,
\]
for some $\epsilon > 0$, and for all $t$. Putting this into the above evolution equation gives
\[
	\ddt (a^2 + \Vert A \Vert^2) \leq - 2 \, \epsilon \, (a^2 + \Vert A \Vert^2)^2,
\]
from which $a, A \to 0$ by comparing with $y' = - 2 \epsilon y^2$. Thus, $\mu_\infty = 0$ in these cases. Regarding cases (i) and (iv), it can be seen that for some initial values the limit is non-zero (take for instance $v_0 = 0$), and for other it is zero (for example, $0 \neq v_0 \in \ker A_0$).

To conclude the proof let us deal now with the case $k > 2$. Assume first that we are in case~ (v), namely $v_0 \notin \Im A_0$.  Since $a$ and $A$ evolve only by scaling, and since $S$ is a homogeneous polynomial in $a$, $A$, it follows that $S$ also evolves only by scaling. In particular its eigenspaces are constant along the flow, and preserved by the ODE for $v$ in \eqref{eqn_gbfaa}. We may thus use the notation introduced in the analysis of the normalized flow, and decompose $v$ as a sum of $S$-eigenvectors. Using that $\tilde v(t) \to \tilde v_\infty \neq 0$ as $t \to \infty$, with $\tilde v_\infty$ an eigenvector of $S_0$ with eigenvalue $\lambda_m = (\tfrac{k}{4} - \unm) a_0^2$, we obtain that the unnormalized solution satisfies
\[
	\frac{r_m}{a^2 + \Vert A \Vert^2} \, \, \, \underset{t\to T}\longrightarrow  \, \, \,  L \neq 0.
\]
But $a^2 / \Vert A \Vert^2$ is constant along the flow, thus we also have
\begin{equation}\label{eqn_rma2}
	\frac{r_m}{a^2}  \, \, \, \underset{t\to T}\longrightarrow  \, \, \,  \tilde L \neq 0.
\end{equation}
On the other hand, the orthogonal projection $v_m$ evolves by
\[
    \ddt v_m =  c \, v_m + S v_m  - \unm \Vert v\Vert^2 \, v_m  = c \, v_m + (\tfrac{k}{4} -\unm) a^2 v_m - \unm \Vert v \Vert^2 v_m = 2 \, c \, v_m,
\]
therefore $r_m = \unm \Vert v_m \Vert^2$  satisfies
\[
    \ddt r_m = 4 \, c \, r_m.
\]
Since $a^4$ satisfies the same linear evolution equation, it is clear that $r_m / a^4$ is constant along the flow. Putting this together with \eqref{eqn_rma2} yields that $a^2$ converges to a non-zero value as $t\to T$. It is in particular bounded, hence so are also $A$ and $v$, and from this it follows that $T = +\infty$, and $\mu(t) \to \mu_\infty \neq 0$ as $t\to\infty$.

Finally, in case (vi) we have $v_0 \in \Im A_0$, and from the normalized flow analysis  we know that
\[
	\frac{c}{a^2} \to \tilde c_\infty > 0.
\]
(Recall that $a^2 / (a^2 + \Vert A \Vert^2)$ remains constant.) In particular, there exists $\epsilon > 0$ such that $c \geq \epsilon a^2$ for all $t$. Using this and \eqref{eqn_gbfaa} we obtain
\[
	\ddt a^2 \geq 2 \, \epsilon \,  a^4,
\]
from which it follows that $T < \infty$ and $\Vert \mu \Vert \to \infty$ as $t\to T$, by comparison with $y' = 2 \epsilon y^2$. This concludes the proof.
\end{proof}

The following is an immediate consequence of the proof of Theorem \ref{thm_aa}:

\begin{corollary}\label{cor_plurisolitons}
Let $\G$ be a non-nilpotent, almost-abelian Lie group endowed with a left-{\allowbreak}invariant SKT structure $(J,g)$, and denote by $\mu = \mu(a,v,A)$ its corresponding bracket. Then, $g$ is a pluriclosed soliton if and only if one of the following holds:
\begin{itemize}
  \item[(i)] $v = 0$;
  \item[(ii)] $v\neq 0$ is an eigenvector of $S = (\tfrac{k}{4} - \unm) a^2 \Id_{\ngo_1} - \unm A A^t + \tfrac{a}{4} (A+A^t)$, with eigenvalue $\lambda = \unm \Vert v \Vert^2$.
\end{itemize}
In case (i), when $a=0$ the soliton is K\"ahler Ricci-flat, and when $a\neq 0$ the soliton is expanding, steady or shrinking, according to whether $k <2$, $k=2$ or $k>2$, respectively. Case (ii) can only occur when $k>2$, and the soliton is steady when $\lambda = (\tfrac{k}{4}-\unm) a^2$ (equivalent to $A v = 0$), and shrinking when $\lambda < (\tfrac{k}{4}-\unm) a^2$.
\end{corollary}

\begin{remark}\label{rem_steadystatic}
A quick computation using Lemma \ref{lem_Pmuaa} shows that all steady solitons in the almost-abelian case are in fact static solutions, i.e.~ they satisfy $P_\mu = 0$.
\end{remark}

By Proposition \ref{prop_genk}, an immediate consequence of Corollary \ref{cor_plurisolitons} is

\begin{corollary}
Pluriclosed  solitons on almost-abelian Lie groups with $k\leq 2$ (this includes the unimodular case) are compatible with a generalized K\"ahler structure.
\end{corollary}

\begin{remark}\label{rem_semialg}
It is not hard to see that for pluriclosed solitons with $Av = 0$, the soliton derivation $D$ is in fact symmetric, cf.~ Remark \ref{rem_algebraic}. On the other hand, for the case when $Av \neq 0$ (which is satisfied for example by the shrinking solitons), $D$  is never symmetric. This can be seen by using Lemma \ref{lem_Pmuaa}, and the fact  that a derivation of a solvable Lie algebra must preserve the nilradical.
\end{remark}

We conclude this section with an example of a Lie group admitting two different types of left-invariant pluriclosed solitons.

\begin{example}\label{exa_shrinksteady}
Let $\mu = \mu(2,0,A)$ be the $10$-dimensional almost-abelian Lie bracket with
\[
    A := \twomatrix{- \Id_{\RR^6}}{0}{0}{0} \in \glg(8,\RR),
\]
the blocks according to a fixed decomposition $\RR^8 = \RR^6 \oplus \RR^2$. Set $J_1$ to be a complex structure on $\RR^8$ respecting said decomposition. Theorem \ref{thm_SKT} implies that the corresponding simply-connected Lie group with left-invariant Hermitian structure $(\G_A, J, g)$ is SKT, with $k = 3$. By Corollary~ \ref{cor_plurisolitons},~ (i), it is furthermore a shrinking pluriclosed soliton.

Consider now $\tilde \mu = \mu(2,v,A)$, with the same $A$, and $v = (0, \ldots, 0,1,1) \in \RR^8$ (six $0$'s). It is easy to see that after Proposition \ref{prop_movingbrackets}, this bracket corresponds to another left-invariant Hermitian structure $(J,\tilde g)$ on the same Lie group $\G_A$ (see \cite[Prop.~4.3]{LW17}), with the same complex structure. Again by Corollary \ref{cor_plurisolitons}, it turns out that $\tilde g$ is a steady pluriclosed soliton. Thus, the group $\G_A$ admits at the same time solitons of distinct type. In particular some solutions of the homogeneous pluriclosed flow will have finite extinction time, whereas some others will exist for all positive times.
\end{example}

\begin{appendix}

\section{Linear algebra estimates}\label{app}

\begin{lemma}\label{lem_appendix}
If $E\in \glg_n(\RR)$ has eigenvalues $ \lambda_1,\ldots,\lambda_n \in \CC$, then
\begin{equation}\label{eqn:ineqElambda}
 \tr E E^t =  \left\Vert E \right\Vert^2 \geq \sum_{i=1}^n \vert \lambda_i \vert^2,
\end{equation}
and equality holds if and only if $E$ is a normal operator.
\end{lemma}
\begin{proof}
Consider on $\glg_n(\RR)$ the negative gradient flow for the functional $\Vert [E,E^t] \Vert^2$, whose integral curves satisfy the ODE
\begin{equation}\label{eqn_odeapp}
    \ddt E = 4 \left[E, \left[E,E^t\right] \right], \qquad E(0) = E_0.
\end{equation}
The solution $E(t)$ is entirely contained in the orbit $\Gl_n(\RR) \cdot E_0 := \{ h \, E_0 \, h^{-1} : h\in \Gl_n(\RR)\}$, because the vector field in \eqref{eqn_odeapp} is always tangent to it. Hence the eigenvalues stay constant. On the other hand, the norm evolves by
\[
    \ddt \Vert E \Vert^2 =  2 \tr E' E^t =   -8 \left\Vert \left[E,E^t\right] \right\Vert^2.
\]
This has two important consequences: firstly, there exists an accumulation point $E_\infty$, which is a fixed point of the system (thus $E_\infty$ is normal), and it is unique by {\L}ojasiewicz' theorem \cite{Loj63},  because the flow is the gradient flow of a polynomial function. Secondly, along the flow the the left-hand-side of \eqref{eqn:ineqElambda} is non-increasing, while the right-hand-side is  constant. Since $E_\infty$ is normal, equality is attained at the limit, and hence at all previous times the strict inequality must hold.
\end{proof}

\begin{corollary}\label{cor_app}
For any $E\in \glg_n(\RR)$ with eigenvalues $\lambda_1, \ldots, \lambda_n \in \CC$ we have that
\[
      \Vert S(E) \Vert^2 \geq \sum_{i=1}^n \Re(\lambda_i)^2,
\]
with equality if and only if $E$ is normal. Here $S(E) = \unm (E+E^t)$ denotes the symmetric part.
\end{corollary}
\begin{proof}
Using Lemma \ref{lem_appendix} we obtain
\[
    \Vert S(E) \Vert^2 = \unm \left(\tr E^2 + \Vert E \Vert^2 \right)  \geq  \unm  \sum_{i=1}^n \left(\lambda_i^2 + \vert \lambda_i \vert^2 \right) = \sum_{i=1}^n \Re(\lambda_i)^2,
\]
where the last identity follows by pairing each non-real eigenvalue with its complex conjugate.
\end{proof}

\end{appendix}

\bibliography{aleklow}

\renewcommand{\MR}[1]{} \newcommand{\noop}[1]{} \def\cprime{$'$}
\providecommand{\bysame}{\leavevmode\hbox to3em{\hrulefill}\thinspace}
\providecommand{\MR}{\relax\ifhmode\unskip\space\fi MR }
\providecommand{\MRhref}[2]{%
  \href{http://www.ams.org/mathscinet-getitem?mr=#1}{#2}
}
\providecommand{\href}[2]{#2}
\begin{thebibliography}{Lau15b}

\bibitem[Bis89]{Bis89}
Jean-Michel Bismut, \emph{A local index theorem for non-{K}\"ahler manifolds},
  Math. Ann. \textbf{284} (1989), no.~4, 681--699. \MR{1006380}

\bibitem[BL17a]{BL17}
Christoph B{\"o}hm and Ramiro~A. Lafuente, \emph{Immortal homogeneous {R}icci
  flows}, Invent. Math. (2017), in press.

\bibitem[BL17b]{GIT}
\bysame, \emph{Real geometric invariant theory}, preprint
  (ar{X}iv:1701.00643v3), 2017.

\bibitem[Bol16]{Boling}
Jess Boling, \emph{Homogeneous solutions of pluriclosed flow on closed complex
  surfaces}, J. Geom. Anal. \textbf{26} (2016), no.~3, 2130--2154. \MR{3511471}

\bibitem[EFV12]{EFV12}
Nicola Enrietti, Anna Fino, and Luigi Vezzoni, \emph{Tamed symplectic forms and
  strong {K}\"ahler with torsion metrics}, J. Symplectic Geom. \textbf{10}
  (2012), no.~2, 203--223. \MR{2926995}

\bibitem[EFV15]{EFV15}
\bysame, \emph{The pluriclosed flow on nilmanifolds and tamed symplectic
  forms}, J. Geom. Anal. \textbf{25} (2015), no.~2, 883--909.

\bibitem[Enr13]{Enr13}
Nicola Enrietti, \emph{Static {SKT} metrics on {L}ie groups}, Manuscripta Math.
  \textbf{140} (2013), no.~3-4, 557--571. \MR{3019139}

\bibitem[FKV15]{FinoKasuyaVezzoni15}
Anna Fino, Hisashi Kasuya, and Luigi Vezzoni, \emph{S{KT} and tamed symplectic
  structures on solvmanifolds}, Tohoku Math. J. (2) \textbf{67} (2015), no.~1,
  19--37. \MR{3337961}

\bibitem[FOU15]{FinoOtalUgarte}
Anna Fino, Antonio Otal, and Luis Ugarte, \emph{Six-dimensional solvmanifolds
  with holomorphically trivial canonical bundle}, Int. Math. Res. Not. IMRN
  (2015), no.~24, 13757--13799. \MR{3436163}

\bibitem[FPS04]{FinoPartonSalamon}
Anna Fino, Maurizio Parton, and Simon Salamon, \emph{Families of strong {KT}
  structures in six dimensions}, Comment. Math. Helv. \textbf{79} (2004),
  no.~2, 317--340. \MR{2059435}

\bibitem[FT09]{FT09}
Anna Fino and Adriano Tomassini, \emph{Non-{K}\"ahler solvmanifolds with
  generalized {K}\"ahler structure}, J. Symplectic Geom. \textbf{7} (2009),
  no.~2, 1--14. \MR{2496412}

\bibitem[FV15]{FV15}
Anna Fino and Luigi Vezzoni, \emph{Special {H}ermitian metrics on compact
  solvmanifolds}, J. Geom. Phys. \textbf{91} (2015), 40--53. \MR{3327047}

\bibitem[Gau97]{Gau97}
Paul Gauduchon, \emph{Hermitian connections and {D}irac operators}, Boll. Un.
  Mat. Ital. B (7) \textbf{11} (1997), no.~2, suppl., 257--288. \MR{1456265}

\bibitem[Gua14]{Gua14}
Marco Gualtieri, \emph{Generalized {K}\"ahler geometry}, Comm. Math. Phys.
  \textbf{331} (2014), no.~1, 297--331. \MR{3232003}

\bibitem[Lau03]{Lau2003}
Jorge Lauret, \emph{On the moment map for the variety of {L}ie algebras}, J.
  Funct. Anal. \textbf{202} (2003), no.~2, 392--423. \MR{1990531}

\bibitem[Lau11]{nilRF}
\bysame, \emph{The {R}icci flow for simply connected nilmanifolds}, Comm. Anal.
  Geom. \textbf{19} (2011), no.~5, 831--854.

\bibitem[Lau12]{spacehm}
\bysame, \emph{Convergence of homogeneous manifolds}, J. London Math. Soc.
  \textbf{86} (2012), 701--727.

\bibitem[Lau13]{homRF}
\bysame, \emph{Ricci flow of homogeneous manifolds}, Math Z. \textbf{274}
  (2013), 373--403.

\bibitem[Lau15a]{Lau15}
\bysame, \emph{Curvature flows for almost-hermitian {L}ie groups}, Trans. Amer.
  Math. Soc. \textbf{367} (2015), no.~10, 7453--7480. \MR{3378836}

\bibitem[Lau15b]{Lau15b}
\bysame, \emph{Geometric flows and their solitons on homogeneous spaces}, arXiv
  preprint arXiv:1507.08163 (2015).

\bibitem[Loj63]{Loj63}
Stanislaw Lojasiewicz, \emph{Une propri{\'e}t{\'e} topologique des
  sous-ensembles analytiques r{\'e}els}, Les {\'e}quations aux d{\'e}riv{\'e}es
  partielles \textbf{117} (1963), 87--89.

\bibitem[Lot07]{Lot07}
John Lott, \emph{On the long-time behavior of type-{III} {R}icci flow
  solutions}, Math. Ann. \textbf{339} (2007), no.~3, 627--666.

\bibitem[LW17]{LW17}
Jorge Lauret and Cynthia Will, \emph{On the symplectic curvature flow for
  locally homogeneous manifolds}, J. Symplectic Geom. \textbf{15} (2017),
  no.~1, 1--49. \MR{3652072}

\bibitem[MS11]{MadsenSwann}
Thomas~Bruun Madsen and Andrew Swann, \emph{Invariant strong {KT} geometry on
  four-dimensional solvable {L}ie groups}, J. Lie Theory \textbf{21} (2011),
  no.~1, 55--70. \MR{2797819}

\bibitem[PV17]{PV17}
Mattia Pujia and Luigi Vezzoni, \emph{A remark on the bismut-ricci form on
  2-step nilmanifolds}, arXiv preprint arXiv:1707.08809 (2017).

\bibitem[ST10]{ST10}
Jeffrey Streets and Gang Tian, \emph{A parabolic flow of pluriclosed metrics},
  Int. Math. Res. Not. IMRN (2010), no.~16, 3101--3133. \MR{2673720}

\bibitem[ST11]{ST11}
\bysame, \emph{Hermitian curvature flow}, Journal of the European Mathematical
  Society \textbf{13} (2011), no.~3, 601--634.

\bibitem[ST12]{ST12}
\bysame, \emph{Generalized {K}\"ahler geometry and the pluriclosed flow},
  Nuclear Phys. B \textbf{858} (2012), no.~2, 366--376. \MR{2881439}

\bibitem[ST13]{Str13}
\bysame, \emph{Regularity results for pluriclosed flow}, Geometry \& Topology
  \textbf{17} (2013), no.~4, 2389--2429.

\bibitem[Str16]{Str16}
Jeffrey Streets, \emph{Pluriclosed flow, born-infeld geometry, and rigidity
  results for generalized k{\"a}hler manifolds}, Communications in Partial
  Differential Equations \textbf{41} (2016), no.~2, 318--374.

\bibitem[Uga07]{Ugarte07}
Luis Ugarte, \emph{Hermitian structures on six-dimensional nilmanifolds},
  Transform. Groups \textbf{12} (2007), no.~1, 175--202. \MR{2308035}

\bibitem[Vez13]{Vez13}
Luigi Vezzoni, \emph{A note on canonical {R}icci forms on {$2$}-step
  nilmanifolds}, Proc. Amer. Math. Soc. \textbf{141} (2013), no.~1, 325--333.
  \MR{2988734}

\end{thebibliography}
\bibliographystyle{amsalpha}

%

\end{document}